\documentclass{amsart}

\usepackage[utf8]{inputenc}
\usepackage[mathscr]{eucal}

\usepackage{amssymb}
\usepackage{amsmath}
\usepackage{amsthm}
\usepackage{enumitem}

\usepackage{pgf,tikz}
\usetikzlibrary{arrows,matrix}

\numberwithin{equation}{section}
\theoremstyle{plain}
\newtheorem{thm}{Theorem}[section]
\newtheorem{lemma}[thm]{Lemma}
\newtheorem{prop}[thm]{Proposition}
\newtheorem{cor}[thm]{Corollary}
\newtheorem{claim}[thm]{Claim}
\newtheorem{fact}[thm]{Fact}
\newtheorem{open}[thm]{Question}

\theoremstyle{definition}
\newtheorem{defn}[thm]{Definition}

\newtheorem{remark}[thm]{Remark}

\usepackage{xspace} 

\DeclareMathOperator{\bip}{Bip}

\def\Bip{{\rm Bip\,}}

\def\Pic{{\rm Pic}}
\def\Div{{\rm Div}}

\newcommand\wapr{\approx_W}

\renewcommand{\qedsymbol}{$\blacksquare$}

\definecolor{light-gray}{gray}{0.8}
\definecolor{v}{rgb}{0.28,0,0.72}
\definecolor{e}{rgb}{0,1,0.2}
\definecolor{r}{rgb}{1,0,0}

\begin{document}

\title[The sandpile group for trinities]{The sandpile group of a trinity and a canonical definition for the planar Bernardi action}

\author{Tam\'as K\'alm\'an}\thanks{TK was supported by a Japan Society for the Promotion of Science (JSPS) Grant-in-Aid for Scientific Research C (no.\ 17K05244).} 
\address{Department of Mathematics\\
Tokyo Institute of Technology\\
H-214, 2-12-1 Ookayama, Meguro-ku, Tokyo 152-8551, Japan}
\email{kalman@math.titech.ac.jp}

\author{Seunghun Lee}\thanks{SL was supported by Basic Science Research Program through the National Research Foundation of Korea (NRF) funded by the Ministry of Education (NRF-2016R1D1A1B03930998). SL is grateful to the CAMPUS Asia program from which he was supported to visit Tokyo Institute of Technology as an exchange student.}
\address{Department of Mathematical Sciences, Korea Advanced Institute of Science and Technology, 291 Daehak-ro, Yuseong-gu, Daejeon 34141, South Korea}
\email{prosolver@kaist.ac.kr}

\author{Lilla T\'othm\'er\'esz}\thanks{LT was supported by the NSF grant DMS-1455272 and the National Research, Development and Innovation Office of Hungary -- NKFIH, grants no.\ 128673 and 132488.}
\address{Cornell University,
	Ithaca, New York 14853-4201, USA, 
	MTA-ELTE Egerv\'ary Research Group, P\'azm\'any P\'eter s\'et\'any 1/C, Budapest, Hungary}
\email{tmlilla@cs.elte.hu}

\date{}

\begin{abstract}
Baker and Wang define the so-called Bernardi action of the sandpile group of a ribbon graph on the set of its spanning trees. This potentially depends on a fixed vertex of the graph but it is independent of the base vertex if and only if the ribbon structure is planar, moreover, in this case the Bernardi action is compatible with planar duality. Earlier, Chan, Church and Grochow and Chan, Glass, Macauley, Perkinson, Werner and Yang proved analogous results about the rotor-routing action. Baker and Wang moreover showed that the Bernardi and rotor-routing actions coincide for plane graphs.

We clarify
this still confounding
picture by giving a canonical definition for the planar Bernardi/rotor-routing action, and also a canonical isomorphism between sandpile groups of planar dual graphs. Our canonical definition implies the compatibility with planar duality via an extremely short argument.
We also show hidden symmetries of the problem by proving our results in the slightly more general setting of balanced plane digraphs.

Any balanced plane digraph gives rise to a trinity, i.e., a triangulation of the sphere with a three-coloring of the $0$-simplices. Our most important tool is a group associated to trinities, introduced by Cavenagh and Wanless, and a result of a subset of the authors characterizing the Bernardi bijection in terms of a dissection of a root polytope.
\end{abstract}

\maketitle

\section{Introduction}

    For an undirected graph, the sandpile group is a finite Abelian group whose order equals the number of spanning trees of the graph. 
    Free transitive group actions of the sandpile group on the spanning trees have recently been an active topic of investigation. Two such classes of group actions are the rotor-routing actions and the Bernardi actions \cite{Holroyd,Baker-Wang}. 
    In \cite{Chan15,rr_comp,Baker-Wang}, some remarkable and somewhat mysterious properties of these group actions were uncovered.

    Both group actions are defined using the same auxiliary data, namely a ribbon structure and a fixed vertex (which is called the base point). A ribbon structure is a cyclic ordering of the edges around each vertex. If a graph is embedded into an orientable surface, the embedding induces a ribbon structure using the positive orientation of the surface, and conversely, for any ribbon structure there exists a closed orientable surface of minimal genus so that the graph embeds into it, inducing the particular ribbon structure. 
	
	Chan, Church and Grochow \cite{Chan15} prove that the rotor-routing action is independent of the base point if and only if the ribbon structure is planar (that is, the graph is embedded into the plane). Chan, Glass, Macauley, Perkinson, Werner and Yang \cite{rr_comp} show that moreover, in the planar case, the rotor-routing action is compatible with planar duality in the following sense: The sandpile groups of a plane graph and its dual are known to be canonically isomorphic \cite{cori-rossin}, 
	and there is a canonical bijection between the spanning trees of the two graphs. The two bijections intertwine the two rotor-routing actions.
	
	Baker and Wang prove analogous results about the Bernardi action, i.e., that the Bernardi action is independent of the base point if and only if the ribbon structure is planar, and in the planar case, the Bernardi action is compatible with planar duality. They also show that in the planar case, the Bernardi and rotor-routing actions coincide, and exhibit an example that in the general case, the two actions can be different.

    These phenomena still feel somewhat myterious. We contribute to the better understanding of this topic by the followings.
    
    \begin{enumerate}
    	\item We give a canonical (that is, base-point free) definition for the planar Bernardi/rotor-routing action, essentially as an action of a group on one of its cosets, where the group is canonicaly isomorphic to the sandpile group and the coset is canonically in bijection with the spanning trees. (See Corollary \ref{cor:Bernardi_can_def}.)
    	\item Using the canonical definition, we give an extremely short proof for the compatibility of the action with planar duality. (See Corollary \ref{cor:compatibility}.)
    	\item We give a canonical (that is, reference-orientation free) definition for the canonical isomorphism between the sandpile group of a planar graph and its dual. (See Theorem \ref{thm:isomorphism_between_sandpile_groups} and Proposition \ref{prop:isomorphism_agree}.)
    	\item We reprove the coincidence of the Bernardi and rotor-routing actions. (See Theorem \ref{thm:rotor_Bernardi_the_same}.)
    	\item We reveal hidden symmetries of the problem by proving all of the above statements in the slightly more general setting of balanced plane digraphs.
    \end{enumerate}
    
	Indeed it turns out that the natural setting in our case is that of balanced plane digraphs. By plane graph, we mean a ribbon graph that is embedded into the plane.
    A ribbon digraph is \emph{balanced} if incoming and outgoing edges alternate in the cyclic ordering around each vertex. Such graphs are automatically Eulerian. Any undirected graph has a \emph{bidirected graph} associated to it, in which each edge is replaced by a pair of oppositely oriented edges. When the undirected graph is embedded in the plane, these pairs of new edges can be arranged so that the resulting plane digraph is balanced. In this sense, balanced plane digraphs generalize plane graphs. 
	
	\subsection{Our method}
	Our method is to analyze trinities, i.e., triangulations of the sphere with a proper three-coloring of the 0-simplices. Any plane graph gives rise, by its first barycentric subdivision, to a trinity with the three color classes corresponding to vertices, edges and regions, but there are also trinities not coming from plane graphs. For the definitions, see Section \ref{ss:trinities}, and for an exampe, see Figure \ref{fig:trinity_from_a_graph}.
	One can associate three (balanced plane) digraphs to a trinity, one on each color class. We call these $D_V,D_E$ and $D_R$. If the trinity is obtained from a plane graph, then $D_V$ is the bidirected version of $G$ and $D_R$ is the bidirected version of the planar dual $G^*$ of $G$. The third directed graph $D_E$ is the common medial graph of $G$ and $G^*$. 
	
	
	Recently, Cavenagh and Wanless \cite{CW} introduced an abelian group $\mathcal A_W$ for trinities that we call the \emph{trinity sandpile group}. This is a certain quotient group of $\mathbb{Z}^{V\cup E\cup R}$. Blackburn and McCourt \cite{BMcC13} related this group to the sandpile groups of the directed graphs associated to the trinity, showing that all three are isomorphic to the torsion subgroup of $\mathcal A_W$. 
	We give an alternative, more natural embedding of the three sandpile groups into the trinity sandpile group. These embeddings yield canonical isomorphisms between the sandpile groups of $D_V$, $D_E$ and $D_R$. Hence as a special case, we get a canonical isomorphism between the sandpile groups of a plane graph and its dual. We show that in this special case this isomorphism agrees with the one given by Cori and Rossin \cite{cori-rossin} (which is defined using a reference orientation).
	
	The sandpile group of $D_E$ is a quotient group of $\mathbb{Z}_0^E$ (which means the subgroup of $\mathbb{Z}^E$ where the sum of the components is $0$). If the trinity comes from a planar graph $G$ (and hence $D_E$ is the medial graph), then (the characteristic vectors of) the spanning trees of $G$ form a coset of the sandpile group of $D_E$. We show that in this case, the Bernardi action is simply the natural group action of the sandpile group of $D_E$ on the coset of spanning trees, composed with the canonical isomorphism between the sandpile groups of $D_E$ and $D_V$. This gives a very simple canonical definition for the Bernardi action. The action of $G^*$ on its spanning trees can once again be described as a natural group action of the sandpile group of $D_E$ on the coset that corresponds to the spanning trees of $G^*$, composed with an isomorphism between the sandpile groups of $D_R$ and $D_E$. As both the natural group actions and the isomorphisms have a simple relationship to each other, we obtain a very simple proof for the compatibility of the Bernardi action with planar duality.

	\begin{figure}
	  \begin{tikzpicture}[-,>=stealth',auto,scale=0.6,thick]
		  	
	    \begin{scope}[shift={(-14,0)}]
	    \draw [-,color=red] (0.6, 1) -- (2,2.7);
	    \draw [-,color=red] (0.6, 1) -- (2,-0.7);
	    \draw [-,color=red] (3.4, 1) -- (2,2.7);
	    \draw [-,color=red] (3.4, 1) -- (2,-0.7);
	    \draw [-,color=red] (0.6, 1) -- (3.4,1);
	    \draw [-,color=blue] (2, 1) -- (2,1.55);
	    \draw [-,color=blue] (2, 1) -- (2,0.45);
	    \draw [-,color=blue] (2.7, 1.85) -- (2,1.55);
	    \draw [-,color=blue] (2.7, 0.15) -- (2,0.45);
	    \draw [-,color=blue] (1.3, 0.15) -- (2,0.45);
	    \draw [-,color=blue] (1.3, 1.85) -- (2,1.55);
	    \draw [-,color=blue,rounded corners=5pt] (1.3, 1.85) -- (1.55, 3) --  (1.9,3.5) -- (2.5,3.5) -- (5,1);
	    \draw [-,color=blue,rounded corners=5pt] (1.3, 0.15) -- (1.55, -1) --  (1.8,-1.5) -- (2.5,-1.5) -- (5,1);
	    \draw [-,color=blue] (2.7, 1.85) -- (5,1);
	    \draw [-,color=blue] (2.7, 0.15) -- (5,1);
	    \draw [color=blue,fill=blue] (0.6, 1) circle [radius=0.1];
	    \draw [color=blue,fill=blue] (2, 2.7) circle [radius=0.1];
	    \draw [color=blue,fill=blue] (3.4, 1) circle [radius=0.1];
	    \draw [color=blue,fill=blue] (2, -0.7) circle [radius=0.1];
	    \draw [color= red,fill=red] (2, 1.55) circle [radius=0.1];
	    \draw [color= red,fill=red] (2, 0.45) circle [radius=0.1];
	    \draw [color= red,fill=red] (5, 1) circle [radius=0.1];
	    \end{scope}
	    
	    \begin{scope}[shift={(-7,0)}]
	    \draw[color=light-gray,fill=light-gray] (0.6, 1) -- (1.3, 1.85) -- (2,1.55) -- cycle;
	    \draw[color=light-gray,fill=light-gray] (0.6, 1) -- (2, 1) -- (2,0.45) -- cycle;
	    \draw[color=light-gray,fill=light-gray] (0.6, 1) -- (1, -1.5) -- (2.2, -2) -- (4, -0.8) -- (5,1) -- (2.5, -1.5) -- (1.8, -1.5) -- (1.55, -1) -- (1.3, 0.15) -- cycle;
	    \draw[color=light-gray,fill=light-gray,rounded corners=5pt] (0.6, 1) -- (1, -2) -- (2.2, -2.5) -- (4, -1) -- (5,1) -- (2.5, -1.5) -- (1.8, -1.5) -- (1.55, -1) -- (1.3, 0.15) -- cycle;
	    \draw[color=light-gray,fill=light-gray] (1.3, 0.15) -- (2, 0.45) -- (2,-0.7) -- cycle;
	    \draw[color=light-gray,fill=light-gray] (5, 1) -- (2.7, 0.15) -- (2,-0.7) -- cycle;
	    \draw[color=light-gray,fill=light-gray] (5, 1) -- (3.4, 1) -- (2.7,1.85) -- cycle;
	    \draw[color=light-gray,fill=light-gray] (5, 1) -- (2, 2.7) -- (1.3, 1.85) -- (1.55, 2.8) --  (1.9,3) -- (2.5,3) -- cycle;
	    \draw[fill=light-gray,color=light-gray,rounded corners=5pt] (5, 1) -- (2, 2.7) -- (1.3, 1.85) -- (1.55, 3) --  (1.9,3.5) -- (2.5,3.5) -- cycle;
	    \draw[color=light-gray,fill=light-gray] (3.4, 1) -- (2, 1.55) -- (2,1) -- cycle;
	    \draw[color=light-gray,fill=light-gray] (2.7, 1.85) -- (2, 2.7) -- (2,1.55) -- cycle;
	    \draw[color=light-gray,fill=light-gray] (2.7, 0.15) -- (3.4, 1) -- (2,0.45) -- cycle;
	    \draw [-,color=red] (0.6, 1) -- (2,2.7);
	    \draw [-,color=red] (0.6, 1) -- (2,-0.7);
	    \draw [-,color=red] (3.4, 1) -- (2,2.7);
	    \draw [-,color=red] (3.4, 1) -- (2,-0.7);
	    \draw [-,color=red] (0.6, 1) -- (3.4,1);
	    \draw [-,color=blue] (2, 1) -- (2,1.55);
	    \draw [-,color=blue] (2, 1) -- (2,0.45);
	    \draw [-,color=blue] (2.7, 1.85) -- (2,1.55);
	    \draw [-,color=blue] (2.7, 0.15) -- (2,0.45);
	    \draw [-,color=blue] (1.3, 0.15) -- (2,0.45);
	    \draw [-,color=blue] (1.3, 1.85) -- (2,1.55);
	    \draw [-,color=blue,rounded corners=5pt] (1.3, 1.85) -- (1.55, 3) --  (1.9,3.5) -- (2.5,3.5) -- (5,1);
	    \draw [-,color=blue,rounded corners=5pt] (1.3, 0.15) -- (1.55, -1) --  (1.8,-1.5) -- (2.5,-1.5) -- (5,1);
	    \draw [-,color=blue] (2.7, 1.85) -- (5,1);
	    \draw [-,color=blue] (2.7, 0.15) -- (5,1);
	    \draw [-,color=green] (0.6, 1) -- (2,1.55);
	    \draw [-,color=green] (3.5, 1) -- (2,1.55);
	    \draw [-,color=green] (2, 2.7) -- (2,1.55);
	    \draw [-,color=green] (0.6, 1) -- (2,0.45);
	    \draw [-,color=green] (3.4, 1) -- (2,0.45);
	    \draw [-,color=green] (2, -0.7) -- (2,0.45);
	    \draw [-,color=green] (3.4, 1) -- (5,1);
	    \draw [-,color=green] (2, 2.7) -- (5,1);
	    \draw [-,color=green] (2, -0.7) -- (5,1);
	    \draw [-,color=green,rounded corners=5pt] (0.6, 1) -- (1, -2) -- (2.2, -2.5) -- (4, -1) -- (5,1);
	    \draw [color=blue,fill=blue] (0.6, 1) circle [radius=0.1];
	    \draw [color=blue,fill=blue] (2, 2.7) circle [radius=0.1];
	    \draw [color=blue,fill=blue] (3.4, 1) circle [radius=0.1];
	    \draw [color=blue,fill=blue] (2, -0.7) circle [radius=0.1];
	    \draw [color=green,fill=green] (1.3, 1.85) circle [radius=0.1];
	    \draw [color=green,fill=green] (1.3, 0.15) circle [radius=0.1];
	    \draw [color=green,fill=green] (2.7, 1.85) circle [radius=0.1];
	    \draw [color=green,fill=green] (2.7, 0.15) circle [radius=0.1];
	    \draw [color=green,fill=green] (2, 1) circle [radius=0.1];
	    \draw [color= red,fill=red] (2, 1.55) circle [radius=0.1];
	    \draw [color= red,fill=red] (2, 0.45) circle [radius=0.1];
	    \draw [color= red,fill=red] (5, 1) circle [radius=0.1];
	    \end{scope}
	    
	    \begin{scope}[shift={(0,0)}]
	    \draw[fill=light-gray] (0.6, 1) -- (1.3, 1.85) -- (2,1.55) -- cycle;
	    \draw[fill=light-gray] (0.6, 1) -- (2, 1) -- (2,0.45) -- cycle;
	    \draw[fill=light-gray] (0.6, 1) -- (1, -1.5) -- (2.2, -2) -- (4, -0.8) -- (5,1) -- (2.5, -1.5) -- (1.8, -1.5) -- (1.55, -1) -- (1.3, 0.15) -- cycle;
	    \draw[color=light-gray,fill=light-gray,rounded corners=5pt] (0.6, 1) -- (1, -2) -- (2.2, -2.5) -- (4, -1) -- (5,1) -- (2.5, -1.5) -- (1.8, -1.5) -- (1.55, -1) -- (1.3, 0.15) -- cycle;
	    \draw[color=light-gray,fill=light-gray] (1.3, 0.15) -- (2, 0.45) -- (2,-0.7) -- cycle;
	    \draw[color=light-gray,fill=light-gray] (5, 1) -- (2.7, 0.15) -- (2,-0.7) -- cycle;
	    \draw[color=light-gray,fill=light-gray] (5, 1) -- (3.4, 1) -- (2.7,1.85) -- cycle;
	    \draw[color=light-gray,fill=light-gray] (5, 1) -- (2, 2.7) -- (1.3, 1.85) -- (1.55, 2.8) --  (1.9,3) -- (2.5,3) -- cycle;
	    \draw[fill=light-gray,color=light-gray,rounded corners=5pt] (5, 1) -- (2, 2.7) -- (1.3, 1.85) -- (1.55, 3) --  (1.9,3.5) -- (2.5,3.5) -- cycle;
	    \draw[color=light-gray,fill=light-gray] (3.4, 1) -- (2, 1.55) -- (2,1) -- cycle;
	    \draw[color=light-gray,fill=light-gray] (2.7, 1.85) -- (2, 2.7) -- (2,1.55) -- cycle;
	    \draw[color=light-gray,fill=light-gray] (2.7, 0.15) -- (3.4, 1) -- (2,0.45) -- cycle;
	    \draw [-,color=red] (0.6, 1) -- (2,2.7);
	    \draw [-,color=red] (0.6, 1) -- (2,-0.7);
	    \draw [-,color=red] (3.4, 1) -- (2,2.7);
	    \draw [-,color=red] (3.4, 1) -- (2,-0.7);
	    \draw [-,color=red] (0.6, 1) -- (3.4,1);
	    \draw [-,color=blue] (2, 1) -- (2,1.55);
	    \draw [-,color=blue] (2, 1) -- (2,0.45);
	    \draw [-,color=blue] (2.7, 1.85) -- (2,1.55);
	    \draw [-,color=blue] (2.7, 0.15) -- (2,0.45);
	    \draw [-,color=blue] (1.3, 0.15) -- (2,0.45);
	    \draw [-,color=blue] (1.3, 1.85) -- (2,1.55);
	    \draw [-,color=blue,rounded corners=5pt] (1.3, 1.85) -- (1.55, 3) --  (1.9,3.5) -- (2.5,3.5) -- (5,1);
	    \draw [-,color=blue,rounded corners=5pt] (1.3, 0.15) -- (1.55, -1) --  (1.8,-1.5) -- (2.5,-1.5) -- (5,1);
	    \draw [-,color=blue] (2.7, 1.85) -- (5,1);
	    \draw [-,color=blue] (2.7, 0.15) -- (5,1);
	    \draw [-,color=green] (0.6, 1) -- (2,1.55);
	    \draw [-,color=green] (3.5, 1) -- (2,1.55);
	    \draw [-,color=green] (2, 2.7) -- (2,1.55);
	    \draw [-,color=green] (0.6, 1) -- (2,0.45);
	    \draw [-,color=green] (3.4, 1) -- (2,0.45);
	    \draw [-,color=green] (2, -0.7) -- (2,0.45);
	    \draw [-,color=green] (3.4, 1) -- (5,1);
	    \draw [-,color=green] (2, 2.7) -- (5,1);
	    \draw [-,color=green] (2, -0.7) -- (5,1);
	    \draw [-,color=green,rounded corners=5pt] (0.6, 1) -- (1, -2) -- (2.2, -2.5) -- (4, -1) -- (5,1);
	    \draw [->,rounded corners=5pt] (1.3, 1.85) -- (2, 3.2) -- (2.7,1.85);
	    \draw [->,rounded corners=5pt] (2.7,1.85) -- (1.3, 1.85);
	    \draw [->,rounded corners=5pt] (1.3, 1.85) -- (2,1);
	    \draw [->,rounded corners=5pt] (2, 1) -- (2.7, 1.85);
	    \draw [->,rounded corners=5pt] (2, 1) -- (1.3, 0.15);
	    \draw [->,rounded corners=5pt] (1.3, 0.15) -- (2.7,0.15);
	    \draw [->,rounded corners=5pt] (2.7, 0.15) -- (2, 1);
	    \draw [->,rounded corners=5pt] (2.7, 1.85) -- (4, 1) -- (2.7, 0.15);
	    \draw [->,rounded corners=5pt] (2.7, 0.15) -- (2, -1.2) -- (1.3, 0.15);
	    \draw [->,rounded corners=5pt] (1.3, 0.15) -- (0, 1) -- (1.3, 1.85);
	    
	    \draw [color=blue,fill=blue] (0.6, 1) circle [radius=0.1];
	    \draw [color=blue,fill=blue] (2, 2.7) circle [radius=0.1];
	    \draw [color=blue,fill=blue] (3.4, 1) circle [radius=0.1];
	    \draw [color=blue,fill=blue] (2, -0.7) circle [radius=0.1];
	    \draw [color=green,fill=green] (1.3, 1.85) circle [radius=0.1];
	    \draw [color=green,fill=green] (1.3, 0.15) circle [radius=0.1];
	    \draw [color=green,fill=green] (2.7, 1.85) circle [radius=0.1];
	    \draw [color=green,fill=green] (2.7, 0.15) circle [radius=0.1];
	    \draw [color=green,fill=green] (2, 1) circle [radius=0.1];
	    \draw [color= red,fill=red] (2, 1.55) circle [radius=0.1];
	    \draw [color= red,fill=red] (2, 0.45) circle [radius=0.1];
	    \draw [color= red,fill=red] (5, 1) circle [radius=0.1];
	    
	    \end{scope}
	  \end{tikzpicture}
		\caption{A plane graph and its dual (left panel), the corresponding trinity (middle panel), and the digraph $D_E$ (right panel).}
		\label{fig:trinity_from_a_graph}  	
	\end{figure}
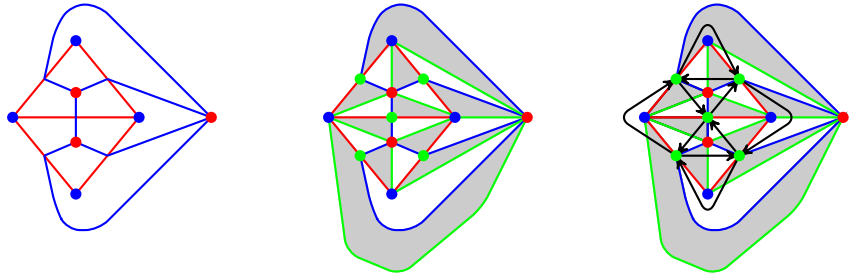
	
	
	By looking at the case of general trinities we can discover so far hidden symmetries, since in this case the roles of the three directed graphs $D_V$, $D_E$ and $D_R$ become completely interchangeable. 
	To be able to proceed in this setting,
	we need to find out what corresponds to the coset of spanning trees in this case. This turns out to the be the notion of \emph{hypertrees}, which first appeared in \cite{hiperTutte,alex}.
	In fact, hypertrees are a common generalization of spanning trees and break divisors. 
	We show that hypertrees form a system of representatives of a coset of the sandpile group of a balanced plane digraph (see Theorem \ref{thm:planar_hypertrees_represent}).
	This is a similar statement to the result by An, Baker, Kuperberg and Shokrieh \cite{ABKS} that break divisors form a system of representatives of a coset of the sandpile group of a graph.
	The two statements are equivalent in the case of plane graphs, but neither generalizes the other. 
	It is an open problem to give a common generalization for these two theorems.
	
	In the case of graphs, Baker and Wang defined the Bernardi action as the composition of the natural action of the sandpile group on break divisors and the Bernardi bijection between break divisors and spanning trees. To be able to examine this picture for balanced plane digraphs, we need a version of the Bernardi bijection that works between hypertrees.
	The version of the Bernardi bijection that we refer to was first defined by K\'alm\'an \cite{hiperTutte} and then recast by K\'alm\'an and Mura\-kami \cite{hitoshi} (relying on fundamental ideas of Postnikov \cite{alex}) in terms of a certain triangulation of the root polytope of a plane bipartite graph. Later it was generalized by K\'alm\'an and T\'othm\'er\'esz \cite{Hyper_Bernardi} to a bijection that works for any ribbon bipartite graph by constructing a certain dissection of its root polytope into simplices. This latter version contains as a special case Baker and Wang's bijection between the spanning trees and the break divisors of a ribbon graph.
	We prove that the Bernardi bijections between hypertrees on the three color classes commute with the natural sandpile actions (see Theorem \ref{thm:Bernardi_compaticle_w_sandpile_action}). 
	One consequence is that the Bernardi action of the sandpile group of a plane graph on its spanning trees agrees with the natural action of the sandpile group of the medial (di)graph, on the same spanning trees, via the natural isomorphism of the sandpile groups.
	
    
\subsection{Outline of the paper}
Section \ref{s:prelim} provides the necessary background. 
Subsection \ref{ss:action} introduces the sandpile group, and surveys the results of Baker and Wang.
Subsection \ref{ss:trinities} describes trinities and the graphs and digraphs associated to them. In Subsection \ref{ss:sandpile_group} we give the definition of the trinity sandpile group. In Subsection \ref{ss:hypertrees} we discuss hypertrees. Subsection \ref{ss:Jaeger} introduces Jaeger trees (another important technical tool for the paper), and the generalized Bernardi bijection. 

In Section \ref{sec:embedding_the_sandpile_group} we establish our embeddings of the sandpile groups into the trinity sandpile group, and we give the (canonical) definition of the  isomorphisms between the sandpile groups associated to a trinity. We also prove that one of these isomorphisms generalizes the canonical isomorphism between the sandpile group of a plane (undirected) graph and its dual, hence we obtain a more natural definition for the latter. (As far as we know, all previous definitions used an arbitrary orientation as auxiliary data.)
In Section \ref{sec:Bernardi_and_sandpile_actions} we show that the Bernardi bijection commutes with the natural sandpile actions, and we deduce a canonical definition for the Bernardi action for balanced plane digraphs. Using the canonical definition, we 
give our
proof for the compatibility of the Bernardi action with planar duality. In Section \ref{sec:representation} we 
show that hypertrees are a set of representatives of a coset of the sandpile group of the appropriate balanced plane digraph.
In Section \ref{sec:rotor} we provide background on rotor-routing, and prove that the rotor-routing action agrees with the Bernardi action for balanced plane digraphs.

{\bf Acknowledgment.} We benefited from conversations with Dylan Thurston. In particular, the idea that hypertrees can be used to represent elements of the sandpile group first occurred to him, before T\'othm\'er\'esz independently discovered the connection between hypertrees and break divisors.

\section{Preliminaries}\label{s:prelim}

\subsection{Basic definitions}
\label{ss:basic}

Throughout this paper, we assume all graphs and digraphs to be connected. We allow loops and multiple edges.
For a digraph $D$, we denote the outdegree of a node $v$ by $d^+(v)$. For two disjoint sets of nodes $U$ and $W$, we denote by $d(U,W)$ the number of directed edges having their tail in $U$ and their head in $W$. In particular, for vertices $u$ and $v$, we let $d(u,v)$ denote the number of edges pointing from $u$ to $v$.

A subgraph of an undirected graph is called a spanning tree if it is connected and cycle-free. A subgraph of a digraph is called an arborescence with root $r$ if we get a tree by forgettig its orientation, and each vertex is reachable on a directed path from vertex $r$.

For an undirected graph, a ribbon structure is the choice of a cyclic ordering of the edges around each vertex. If a graph is embedded into an orientable surface, the embedding gives a ribbon structure using the positive orientation of the surface, and conversely, for any ribbon graph there exists a closed orientable surface of minimal genus so that the graph embeds into it, giving the particular ribbon structure. For us, the most important case is the case of the graphs embedded into the plane (plane graphs). For an edge $xy$ of the graph, we denote by $xy^+$ the edge following $xy$ at $x$ according to the ribbon structure. For a digraph, a ribbon structure is a choice of a cyclic ordering of the union of in- and out-edges around each vertex.

The Laplacian matrix of a digraph is the following matrix $L_D \in \mathbb{Z}^{V \times V}$:
\[
L_D(u,v) = \left\{\begin{array}{cl} -d^+(v) & \text{if } u=v, \\
d(v, u) & \text{if } u\neq v.      
\end{array} \right.
\]

Let us introduce notations for some special vectors in $\mathbb{Z}^{|V|}$.
By $\mathbf{0}$, we denote the vector with all coordinates equal to zero, while by $\mathbf{1}$ the vector with all coordinates equal to one. For a set $S\subseteq V$, we let $\mathbf{1}_S$ denote the characteristic vector of $S$, i.e. $\mathbf{1}_S(v)=1$ for $v\in S$ and $\mathbf{1}_S(v)=0$ otherwise.

\subsection{Sandpile groups and the Bernardi action}\label{ss:sandpile_and_Bernardi}
\label{ss:action}

In this section we give the definition of the sandpile group and outline the results of Baker and Wang about the Bernardi action. As later on we will need the sandpile group of Eulerian digraphs, we give the definition for this broader case.

For an Eulerian digraph $D=(V,A)$, we denote by $\Div(D)$ the free Abelian group on $V$. For $x\in\Div(D)$ and $v\in V$, we use the notation $x(v)$ for the coefficient of $v$. We refer to $x$ as a chip configuration, and to $x(v)$ as the number of chips on $v$. We use the notation $\deg(x)=\sum_{v\in V} x(v)$, and call $\deg(x)$ the \emph{degree} of $x$.
We also write $\Div^d(D)=\{x\in \Div(D) : \deg(x)=d\}$.

We call two chip configurations $x$ and $y$ \emph{linearly equivalent} if there exists $z\in \mathbb{Z}^V$ such that $y=x + L_D z$. We use the notation $x\sim y$ for linear equivalence. Notice that, as for Eulerian digraphs we have $L_D \mathbf{1} = \mathbf{0}$, we can suppose that $z$ has nonnegative elements and $z(v)=0$ for some $v\in V$. Note also that linearly equivalent chip configurations have equal degree. We denote the linear equivalence class of a chip configuration $x$ by $[x]$.

There is an interpretation of linear equivalence using the so-called chip-firing game. In this game, a step consists of firing a node $v$. The firing of $v$ decreases the number of chips on $v$ by the outdegree of $v$, and increases the number of chips on each neighbor $w$ of $v$ by $d(v,w)$. It is easy to check that the firing of $v$ changes $x$ to $x+L_D \mathbf{1}_v$. Hence $x$ is linearly equivalent to $y$ if and only if there is a sequence of firings that transforms $x$ to $y$.

The Picard group of a digraph is the group of chip configurations factorized by linear equivalence: $\Pic(D)=\Div(D)/_\sim$. This is an infinite group. We will be interested in the subgroup corresponding to zero-sum elements, which is called the sandpile group.

\begin{defn}[Sandpile group]
	For an Eulerian digraph $D$, the sandpile group is defined as $\Pic^0(D)=\Div^0(D)/_\sim$.
\end{defn}

It is easy to see that $\Pic(D)=\Pic^0(D)\times \mathbb{Z}$.
The sandpile group is a finite group. We will need the following version of the matrix-tree theorem. 

\begin{fact}\cite{Holroyd}\label{fact:order_of_sandpile_group}
	For an Eulerian digraph $D$, the order of $\Pic^0(D)$ is equal to the number of arborescences rooted at an arbitrary vertex.
\end{fact}

We note that \cite{Holroyd} defines the sandpile group of $D$ as $\mathbb{Z}^{|V|-1}/_{L'_D\mathbb{Z}^{|V|-1}}$ where $L'_D$ is the matrix obtained from the Laplace matrix by deleting the row and column corresponding to a vertex $v$. It is easy to see that this is equivalent to our definition since $L_D\mathbf{1}=0$ and we consider degree zero chip configurations in $\Pic^0(D)$.

We will use the notation $\Pic^d(D)$ for the set of equivalence classes of $\Pic(D)$ consisting of chip configurations of degree $d$.

If we have an undirected graph, we can apply the above definitions to the bidirected version of the graph, that is, where we substitute each undirected edge by two oppositely directed edges. 

Let us turn to the Bernardi action for undirected graphs. For an (undirected) graph $G$, the Bernardi action is an action of $\Pic^0(G)$ on the spanning trees of $G$. To define it, we first need the definition of the Bernardi bijection.

\begin{figure}
	\begin{center}
		\begin{tikzpicture}[-,>=stealth',auto,scale=0.6,
		thick]
		\tikzstyle{o}=[circle,draw]
		\node[o,label=right:{0}] (1) at (8, 0) {$v_1$};
		\node[o,label=above:{1}] (2) at (4, 1.2) {$v_2$};
		\node[o,label=left:{0}] (3) at (0, 0) {$v_3$};
		\node[o,label=below:{1}] (4) at (4, -1.2) {$v_4$};
		\path[every node/.style={font=\sffamily\small}, line width=0.8mm]
		(1) edge node [above] {$e_1$} (2)
		(3) edge node [below] {$e_3$} (4)
		(2) edge node {$e_5$} (4);
		\path[every node/.style={font=\sffamily\small},dashed]
		(4) edge node [below] {$e_4$} (1)
		(2) edge node [above] {$e_2$} (3);
		\end{tikzpicture}
	\end{center}
	\caption{An example of the tour of a spanning tree and the Bernardi bijection. Let the ribbon structure be the one induced by the positive (counterclockwise) orientation of the plane, $b_0=v_1$, and $b_1=v_2$. The tour of the tree of thick edges is $v_1e_1,v_2e_2,v_2e_5,v_4e_3,v_3e_2,$ $v_3e_3,v_4e_4,v_4e_5,v_2e_1,v_1e_4$. The Ber\-nardi bijection gives the break divisor indicated by the numbers.}
	\label{fig:tour_of_T}
\end{figure}
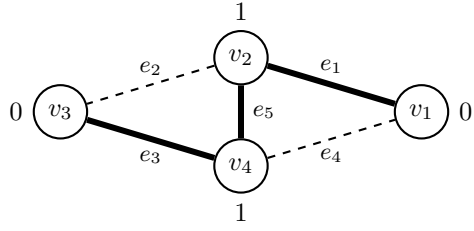 

The Bernardi bijection depends on a ribbon structure of $G$ and on a fixed vertex $b_0$ of $G$ and a fixed edge $b_0b_1$ incident to $b_0$. 
Using this data, to any spanning tree $T$, one can now associate a traversal of the graph, which is called the \emph{tour} of $T$. (This process was introduced by Bernardi \cite{Bernardi_first,Bernardi_Tutte}.)
The tour of $T$ is the following sequence of node-edge pairs: The current node at the first step is $b_0$, and the current edge is $b_0b_1$. If the current node is $x$, the current edge is $xy$, and $xy\notin T$, then the current node of the next step is $x$, and the current edge of the next step is $xy^+$. If the current node is $x$, the current edge is $xy$, and $xy\in T$, then the current node of the next step is $y$, and the current edge of the next step is $yx^+$. The tour stops when $b_0$ would once again become current node with $b_0b_1$ as current edge. (For an example, see Figure \ref{fig:tour_of_T}).
Bernardi proved the following:

\begin{lemma}\label{l:T-tour_cyclic_perm}
	\cite[Lemma 5]{Bernardi_first} In the tour of a spanning tree $T$, each edge $xy$ of $G$ becomes current edge twice, in one case with $x$ as current node, and in the other case with $y$ as current node.
\end{lemma}

For an edge $xy\notin T$, we say that $xy$ is cut through at $x$ during the tour of $T$ if it first becomes current edge with $x$ as current node.

The Bernardi bijection associates a chip configuration to any spanning tree by dropping a chip at the current vertex each time a nonedge of $T$ is cut through (see again Figure \ref{fig:tour_of_T}, also \cite{Baker-Wang}). Bernardi \cite{Bernardi_first} 
and Baker and Wang \cite{Baker-Wang} prove that this is a bijection between the spanning trees of $G$ and the so-called break divisors. We denote this bijection by $\beta_{b_0,b_1}$.

For a graph $G=(V,E)$, a chip configuration $x\in \Div(G)$ is a break divisor if there exists a spanning tree $T$ of $G$ such that $E-T=\{e_1,\dots e_g\}$ and there is a bijection between the edges $\{e_1,\dots e_g\}$ and the chips of $x$ such that each chip sits on one of the endpoints of the edge assigned to it.

Furthermore, it is proved in \cite{ABKS} that break divisors give a system of representatives of $\Pic^{|E|-|V|+1}(G)$.

\begin{thm}\cite{ABKS}\label{thm:break_eredeti_represent}
	For an undirected graph $G=(V,E)$, the set of break divisors form a system of representatives of linear equivalence classes of $\Pic^{|E|-|V|+1}(G)$.
\end{thm}

For a graph $G$, the sandpile group $\Pic^0(G)$ acts on $\Pic^{|E|-|V|+1}(G)$ by addition: For $x\in \Pic^0(G)$ and $z\in \Pic^{|E|-|V|+1}(G)$, we put $x\cdot z= x+z$. Since by Theorem \ref{thm:break_represent}, the break divisors give a system of representatives for $\Pic^{|E|-|V|+1}(G)$, we can think of this natural action as the action of $\Pic^0(G)$ on the break divisors: for $x\in \Pic^0(G)$ and a break divisor $f$, we have $x\cdot f = x \oplus f$, where by $x\oplus f$ we denote the unique break divisor in the linear equivalence class of $x + [f]$, which exists by Theorem \ref{thm:break_eredeti_represent}. We call this group action the \emph{sandpile action}. 

The Bernardi action is defined by pulling the sandpile action of $\Pic^0(G)$ from the break divisors to the spanning trees by using a Bernardi bijection:
$x\cdot T=\beta^{-1}_{b_0,b_1}(x\oplus \beta_{b_0,b_1}(T))$, where $x\in \Pic^0(G)$ and $T$ is a spanning tree of $G$.

Baker and Wang prove that this group action does not depend on the choice of $b_1$, moreover, it is independent of $b_0$ if and only if the ribbon structure is planar. 
For planar graphs, Baker and Wang also prove the compatibility of the Bernardi action with planar duality. Let us explain this statement. 
It is well-known that there is a canonical bijection between the spanning trees of a plane graph and its dual: To any spanning tree $T$ of $G$, we can associate $T^*=\{e^* : e\notin T\}$.
Also, for a planar graph $G$, there is a canonical isomorphism $i:\Pic^0(G) \to \Pic^0(G^*)$ (see \cite{cori-rossin}). 
Baker and Wang proved that the Bernardi action of plane graphs satisfies $(x \cdot T)^* = i(x) \cdot T^*$. 

Let us repeat the definition of $i$ as given in \cite{Baker-Wang}. Let $G$ be a planar undirected graph. We need to fix an orientation $\overrightarrow{e}$ for each edge $e$. Now orient each edge $e^*$ of $G^*$ so that the corresponding edge $\overrightarrow{e}$ of $G$ has to be turned in the negative direction to get the orientation of $\overrightarrow{e}^*$. For an edge $\overrightarrow{e}$ of $G$, let $\delta_{\overrightarrow{e}}\in \mathbb{Z}^V$ be the vector that has coordinate one on the head of $\overrightarrow{e}$, minus one on the tail of $\overrightarrow{e}$, and zero otherwise. For any $g\in \Div^0(G)$, one can find integers $\{a_{\overrightarrow{e}}: e\in E\}$ such that $\sum_{e\in E} a_{\overrightarrow{e}} \delta_{\overrightarrow{e}} = g$. Moreover, two collections of coefficients $\{a_{\overrightarrow{e}}: e\in E\}$ and $\{b_{\overrightarrow{e}}: e\in E\}$ give linearly equivalent chip configurations if and only if $\{a_{\overrightarrow{e}} - b_{\overrightarrow{e}}: e\in E\}$ can be written as the sum of an integer flow in $G$ and an integer flow in $G^*$. Now for $[g] \in \Pic^0(G)$, the image $i([g])$ is defined as $[\sum_{e\in E} a_{\overrightarrow{e}} \delta_{\overrightarrow{e}^*}]$. It can be shown that this is a well defined mapping, which is an isomorphism, and it is independent of the orientation we chose. For more details, see \cite{Baker-Wang} and its references.

In this paper, we analyze the sandpile groups and Bernardi actions of plane graphs (and more generally, balanced plane digraphs) by examining trinities. We give a canonical (orientation-free) definition to the above isomorphism $i$, and also a canonical definition for the Bernardi action of balanced plane digraphs. This definition yields a very short proof for the compatibility of the Bernardi action with planar duality. 

\subsection{Trinities}\label{ss:trinities}

See Figure \ref{f:trinity_ex} and Figure \ref{fig:trinity_from_a_graph} for examples (drawn in the plane) of the following notion.

\begin{defn}[Trinity]
A \emph{trinity} is a triangulation of the sphere $S^2$ together with a three-coloring of the $0$-simplices. (I.e., $0$-simplices joined by a $1$-simplex have different colors.) According to dimension, we will refer to the simplices as \emph{points, edges}, and \emph{triangles}.
\end{defn}

We will use the names red, emerald, and violet for the colors in the trinity and denote the respective sets of points by $R$, $E$, and $V$. Let us color each edge in the triangulation with the color that does not occur among its ends. Then $E$ and $V$ together with the red edges form a bipartite graph that we will call the \emph{red graph} and denote it by $G_R$. Each region of the red graph contains a unique red point. Likewise, the \emph{emerald graph} $G_E$ has red and violet points, emerald edges, and regions marked with emerald points. Finally, the \emph{violet graph} $G_V$ contains $R$ and $E$ as vertices, violet edges, and a violet point in each of its regions.

There are two types of triangles in a trinity: the red, emerald and violet nodes either follow each other in clockwise or counterclockwise order. Let us color a triangle white if the order is clockwise and let us color it black otherwise. Then any two triangles sharing an edge have different colors.

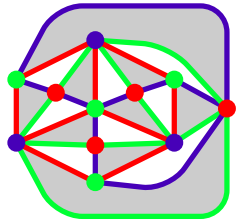
\begin{figure}
	\begin{center}
		\begin{tikzpicture}[-,>=stealth',auto,scale=0.35,
		thick]
		\draw[fill=light-gray,color=light-gray] (0, 6.2) -- (2.8, 6) -- (5, 3.6) -- (4.6, 7) -- (-1, 7) -- (-3, 4.7) -- cycle;
		\draw[color=light-gray,fill=light-gray,rounded corners=10pt] (0, 6.2) -- (2.8, 6) -- (5, 3.6) -- (5, 7.5) -- (-1.5, 7.5) -- (-3, 4.7) -- cycle;
		\draw[color=light-gray,fill=light-gray,rounded corners=10pt] (0, 0.8) -- (2.8, 0.5) -- (5, 3.6) -- (5, -0.5) -- (-1.5, -0.5) -- (-3, 2.3) -- cycle;
		\draw[color=light-gray,fill=light-gray] (0, 0.8) -- (2.8, 0.5) -- (5, 3.6) -- (4.7, 0) -- (-1, 0) -- (-3, 2.3) -- cycle;
		\draw[fill=light-gray] (-3, 4.7) -- (-1.5, 4.2) -- (-3, 2.3) -- cycle;
		\draw[fill=light-gray] (0, 3.6) -- (-1.5, 4.2) -- (0, 6.2) -- cycle;
		\draw[fill=light-gray] (0, 3.6) -- (-3, 2.3) -- (0,2.2) -- cycle;
		\draw[fill=light-gray] (0, 0.8) -- (3, 2.3) -- (0,2.2) -- cycle;
		\draw[fill=light-gray] (0, 3.6) -- (3, 2.3) -- (1.5, 4.2) -- cycle;
		\draw[fill=light-gray] (3, 4.7) -- (0, 6.2) -- (1.5, 4.2) -- cycle;
		\draw[fill=light-gray] (3, 4.7) -- (3, 2.3) -- (5,3.6) -- cycle;
		\draw[-, line width=0.7mm, color=e] (1.5, 4.2) to (3, 2.3);
		\draw[-, line width=0.7mm, color=v] (1.5, 4.2) to (3, 4.7);
		\draw[-, line width=0.7mm, color=e] (1.5, 4.2) to (0, 6.2);
		\draw[-, line width=0.7mm, color=v] (1.5, 4.2) to (0, 3.6);
		\draw[-, line width=0.7mm, color=v] (-1.5, 4.2) to (0, 3.6);
		\draw[-, line width=0.7mm, color=e] (-1.5, 4.2) to (0, 6.2);
		\draw[-, line width=0.7mm, color=v] (-1.5, 4.2) to (-3, 4.7);
		\draw[-, line width=0.7mm, color=e] (-1.5, 4.2) to (-3, 2.3);
		\draw[-, line width=0.7mm, color=v] (0, 2.2) to (0, 3.6);
		\draw[-, line width=0.7mm, color=e] (0, 2.2) to (-3, 2.3);
		\draw[-, line width=0.7mm, color=v] (0, 2.2) to (0, 0.8);
		\draw[-, line width=0.7mm, color=e] (0, 2.2) to (3, 2.3);
		\draw[-, line width=0.7mm, color=e] (5, 3.6) to (3, 2.3);
		\draw[-, line width=0.7mm, color=v] (3, 4.7) to (5, 3.6);
		\draw[-, line width=0.7mm, color=e,rounded corners=10pt] (5, 3.6) -- (2.8, 6) -- (0, 6.2);
		\draw[-, line width=0.7mm, color=v, rounded corners=10pt] (-3, 4.7) -- (-1.5, 7.5) -- (5, 7.5) -- (5, 3.6);
		\draw[-, line width=0.7mm, color=v,rounded corners=10pt] (5, 3.6) -- (2.8, 0.5) -- (0, 0.8);
		\draw[-, line width=0.7mm, color=e, rounded corners=10pt] (-3, 2.3) -- (-1.5, -0.5) -- (5, -0.5) -- (5, 3.6);
		\path[every node/.style={font=\sffamily\small},color=red,line width=0.7mm]
		(-3, 4.7) edge node {} (0, 6.2)
		(3, 2.3) edge node {} (3, 4.7)
		(3, 2.3) edge node {} (0, 3.6)
		(0, 0.8) edge node {} (-3, 2.3)
		(3, 2.3) edge node {} (0, 0.8)
		(-3, 2.3) edge node {} (-3, 4.7)
		(0, 6.2) edge node {} (3, 4.7)
		(0, 3.6) edge node {} (-3, 2.3)
		(0, 6.2) edge node {} (0, 3.6);
		\draw [color=e,fill=e] (0, 0.8) circle [radius=0.3];
		\draw [color=v,fill=v] (3, 2.3) circle [radius=0.3];
		\draw [color=v,fill=v] (-3, 2.3) circle [radius=0.3];
		\draw [color=e,fill=e] (3, 4.7) circle [radius=0.3];
		\draw [color=e,fill=e] (-3, 4.7) circle [radius=0.3];
		\draw [color=v,fill=v] (0, 6.2) circle [radius=0.3];
		\draw [color=e,fill=e] (0, 3.6) circle [radius=0.3];		
		\draw [color=r,fill=r] (5, 3.6) circle [radius=0.3];
		\draw [color=r,fill=r] (1.5, 4.2) circle [radius=0.3];
		\draw [color=r,fill=r] (-1.5, 4.2) circle [radius=0.3];
		\draw [color=r,fill=r] (0, 2.2) circle [radius=0.3];
		\end{tikzpicture}
	\end{center}
	\caption{A trinity}
	\label{f:trinity_ex}
\end{figure}

We can also associate three directed graphs $D_V, D_E$ and $D_R$ to a trinity: The node set of $D_V$ is $V$, and a directed edge points from $v_1\in V$ to $v_2\in V$ if a black triangle incident to $v_1$ and a white triangle incident to $v_2$ share their violet edge. Note that the outdegree of a node is equal to the number of black triangles incident to it, and the indegree of a node is equal to the number of white triangles incident to it. Hence $D_V$ is Eulerian, moreover, it is embedded into the plane so that around each node, in- and out-edges alternate (as the white and the black triangles also alternate). We call such embeddings \emph{balanced}.
We define $D_E$ and $D_R$ similarly, and these are also balanced plane (hence Eulerian) digraphs by the same argument.

As an important special case, we can construct a trinity from a plane graph $G$ in the following way. (See Figure \ref{fig:trinity_from_a_graph} for an example.) Let $V$ be the set of vertices of the graph, and subdivide each edge by a new emerald node. The resulting bipartite graph is $G_R$. Then place a red node in the interior of each region of the plane graph. 
Traverse the boundary of each region of $G_R$ and at each corner of the boundary, connect the emerald or violet node to the red node of the region. Notice that in this case $G_V$ can be obtained from $G^*$, the planar dual of $G$, by subdividing each edge with an emerald node. Moreover, the directed graphs $D_V$ and $D_R$ can be obtained from $G$ and $G^*$ respectively by substituting each edge with two oppositely directed edges. 

In general, it is easy to check that digraphs arising as $D_V$ for a trinity are exactly the balanced plane digraphs. We already pointed out that for a trinity, $D_V$ is a balanced plane digraph. Moreover, for a balanced embedding of a digraph, boundaries of all regions are oriented cycles, and they can be two-colored with respect to the orientation of the cycle. It is easy to check that if we place a red node in all clockwise oriented regions and an emerald node in all counterclockwise oriented regions, moreover, connect each red and emerald node to all violet nodes along the boundary of their respective region, finally connect a red and a violet node if they occupy neighboring regions, then we obtain a trinity one of whose balanced digraphs is the one we started with.

\subsection{The trinity sandpile group}\label{ss:sandpile_group}

Since there are three (planar) digraphs $D_V$, $D_E$, and $D_R$ associated to a trinity, there are also three sandpile groups naturally associated to trinities: $\Pic^0(D_V)$, $\Pic^0(D_E)$, and $\Pic^0(D_R)$. It will turn out that these three groups are isomorphic and we will obtain natural isomorphisms between them using a group that we call the trinity sandpile group, and which appeared first in \cite{CW}. First we need some preparation.

\begin{defn}[$\mathcal{A}$]
	Let $\mathcal{A}$ be the free Abelian group on the set $V\cup E \cup R$. We describe the elements of $\mathcal{A}$ by vector triples $(x_V,x_E,x_R)$, where $x_V\in\mathbb{Z}^V$, $x_E\in\mathbb{Z}^E$, and $x_R\in\mathbb{Z}^R$.
\end{defn}

\begin{defn}[white triangle equivalence]
	Two elements of $\mathcal{A}$ are said to be white triangle equivalent if their difference can be written as an integer linear combination of characteristic vectors of white triangles. We denote white triangle equivalence by $\wapr$.
\end{defn}

Note that $\wapr$ is indeed an equivalence relation.
Now one can define a group by factorizing with white triangle equivalence.

\begin{defn}[$\mathcal{A}_W$, \cite{CW}]
	$\mathcal{A}_W=\mathcal{A}/_{\wapr}$.
\end{defn}

Blackburn and McCourt proved that $\mathcal{A}_W$ is isomorphic to the direct product of 
$\mathbb{Z}^2$ and $\Pic^0(D_V)$. In Section \ref{sec:embedding_the_sandpile_group}, we give a very simple and natural embedding of the sandpile groups $\Pic^0(D_V), \Pic^0(D_E)$ and $\Pic^0(D_R)$ as a subgroup of $\mathcal{A}_W$, which also yields combinatorially nice isomorphisms between these sandpile groups.
From now on, we will call $\mathcal{A}_W$ the trinity sandpile group.

\subsection{Hypertrees}\label{ss:hypertrees}

The following notion will be very important for us. 
It appeared first in \cite{alex} (as a `draconian sequence' or a `degree vector'), then again in \cite{hiperTutte}.

\begin{defn} \cite{alex,hiperTutte}
\label{def:hypertree}
Let $H$ be a bipartite graph and $U$ one of its vertex classes. We say that the vector $f\colon U\to\mathbb{Z}_{\ge0}$ is a \emph{hypertree} on $U$ if there exists a spanning tree $T$ of $H$ that has degree $d_T(u)=f(u)+1$ at each node $u\in U$.
We denote the set of all hypertrees of $H$ on $U$ by $B_{U}(H)$.
\end{defn}

For a spanning tree $T$ of the bipartite graph $H$, we denote by $f_U(T)$ the hypertree on $U$ \emph{realised} or \emph{induced} by $T$, i.e.,
$$
f_U(T)(u)=d_T(u)-1  \quad \forall u\in U.
$$ 

The name hypertree comes from the fact that hypertrees generalize (characteristic vectors of) spanning trees from graphs to hypergraphs in the following sense. A bipartite graph $H$ always induces a hypergraph where one partite class of the bipartite graph corresponds to the vertices of the hypergraph, the other partite class corresponds to the hyperedges, and the edges of $H$ correspond to containment. If we think of $U$ as the set of hyperedges, then a hypertree is a function assigning multiplicities to hyperedges. In the special case where the hypergraph is a graph $G$ (i.e., the bipartite graph $H$ is obtained by subdividing each edge of $G$ by a new point), and $U$ is the partite class of the subdividing points (i.e., it corresponds to the edges of $G$), then the hypertrees on $U$ are exactly the characteristic vectors of the spanning trees of $G$ (cf.\ \cite[Remark 3.2]{hiperTutte}). 

In particular, if a trinity is derived from a plane graph $G$, then the hypertrees of $G_R$ on $E$ are exactly the characteristic vectors of the spanning trees of $G$, while the hypertrees of $G_V$ on $E$ are exactly the characteristic vectors of the spanning trees of $G^*$.

It is also fruitful to think of hypertrees as chip configurations. Indeed, it turns out that the notion of break divisors is in fact a special case of hypertrees.

\begin{prop}\label{prop:hypertree_break_div_relationship}
	For a graph $G=(V,E)$,
	$x\in \mathbb{Z}^V$ is a break divisor on $G$ if and only if $d_G-\mathbf{1}-x$ is a hypertree of $\bip G$ on $V$ where $\bip G$ is obtained from $G$ by subdividing each edge with a new point, and $d_G$ is the vector assigning its degree to each vertex of $G$.
\end{prop}

\begin{proof} 
	Let us call the ``subdividing'' nodes of $\bip G$ the \emph{emerald} nodes. These are in bijection with the edges of $G$. Let us call the original vertices of $G$ the \emph{violet nodes}.
	If $x$ is a break divisor then let us take a spanning tree $T$ of $G$ witnessing this fact. Let $T'$ be the spanning tree of $\bip G$ where we take both edges incident to an emerald node corresponding to an edge of $T$, and for an emerald node corresponding to a nonedge $e$ of $T$, we take the edge incident to $e$ which leads to the violet endpoint of $e$ not containing its chip.
	It is easy to see that $T'$ is a spanning tree of $\bip G$, moreover, its degree sequence on $V$ is $d_G-x$, which shows that $d_G-\mathbf{1}-x$ is a hypertree.
	
	Conversely, if we have a hypertree $h$ of $\bip G$ on $V$, we may consider a realizing spanning tree $T'$ of $\bip G$. It is clear that for the emerald nodes that have degree $2$ in $T'$, the corresponding edges of $G$ form a spanning tree $T$, and $T$ witnesses that $d_G-\mathbf{1}-h$ is a break divisor.
\end{proof}

Theorem \ref{thm:break_eredeti_represent} can be restated in terms of hypertrees in the following form.

\begin{thm}\cite{ABKS}\label{thm:break_represent}
	For any (not necessarily planar) graph $G=(V,E)$, let $\bip G$ be the bipartite graph where we subdivide each edge with a new node. Let $V$ be the set of nodes corresponding to the original vertices of $G$. Then the set of hypertrees of $\bip G$ on $V$ forms a system of representatives of the linear equivalence classes of $\Pic^{|E|-1}(G)$. \hfill\qedsymbol
\end{thm}

In this paper we prove that an analogue of this theorem holds for planar trinities.

\begin{thm}\label{thm:planar_hypertrees_represent}
For a planar trinity, the set of hypertrees of $G_R$ on $V$ gives a system of representatives of linear equivalence classes of $\Pic^{|E|-1}(D_V)$. In other words, for any chip configuration $x_V$ on $V$ with $\deg(x_V)=|E|-1$, there is exactly one hypertree $f\in B_V(G_R)$ such that $f\sim x_V$ (where linear equivalence is meant for the graph $D_V$).
\end{thm}

We prove Theorem \ref{thm:planar_hypertrees_represent} in Section \ref{sec:representation}. 
The statements of Theorems \ref{thm:break_represent} and \ref{thm:planar_hypertrees_represent} are equivalent for planar graphs, but none of them generalizes the other. So far no common generalization is known for the two theorems.

\begin{open}
	Is there a common generalization of Theorems \ref{thm:break_represent} and \ref{thm:planar_hypertrees_represent}?
\end{open}

\subsection{Realizing hypertrees: Jaeger trees and the Bernardi process}\label{ss:Jaeger}

In general, a hypertree can be realized by many different spanning trees of the bipartite graph, and generally we do not care about the representatives. It is sometimes useful, however, to have a nice set of representing spanning trees for the hypertrees. As explained in \cite{Hyper_Bernardi}, the so-called Jaeger trees give us such a nice set of representatives. 

Let us give the definition of Jaeger trees. Let $H$ be a bipartite graph with vertex classes $V$ and $E$. We will once again call the elements of $V$ violet nodes and the elements of $E$ emerald nodes. We will use the notion of the tour of a spanning tree (see Subsection \ref{ss:sandpile_and_Bernardi}), for which we need a fixed ribbon structure of $H$, a fixed node $b_0\in V\cup E$ and a fixed edge $b_0b_1$ incident to $b_0$. 
 
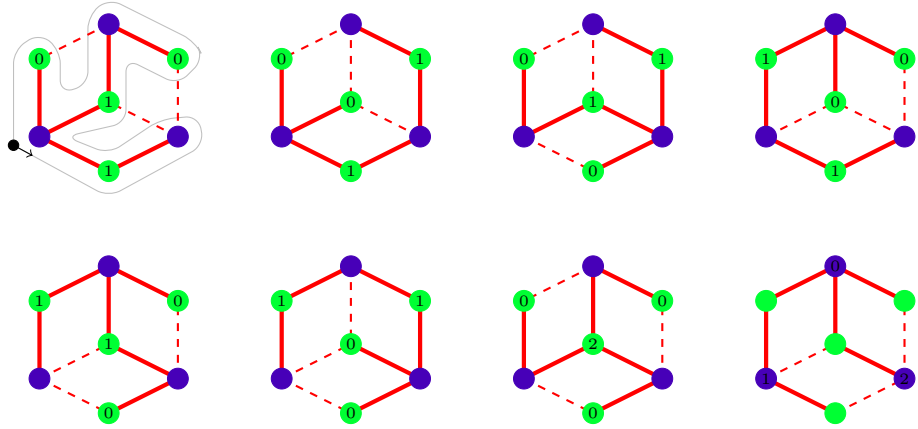
\begin{figure}[h]
	\begin{tikzpicture}[scale=.23]
	
	\begin{scope}[shift={(-14,0)}]
	\draw[-, lightgray, rounded corners=5pt] (12.5,1.5) -- (12.5, 3) -- (12.5, 7) -- (14, 8) -- (15.2, 7) -- (15.2, 4.5) -- (16.8, 5) -- (16.8, 9) -- (18,10) -- (23, 7.5) -- (23.5, 6.5) -- (22, 5) -- (19, 6.5) -- (19, 2.8) -- (15.5, 2) -- (18, 1) -- (21, 3) -- (23.2, 3.2) -- (23.5, 1.5) -- (18, -1.5) -- (12.5, 1.5);
	\draw [dashed, thick,red] (18, 8.5) -- (14,6.5);
	\draw [ultra thick,red] (14, 6.5) -- (14,2);
	\draw [ultra thick,red] (18, 8.5) -- (22,6.5);
	\draw [ultra thick,red] (18, 8.5) -- (18,4);
	\draw [ultra thick,red] (14, 2) -- (18,4);
	\draw [ultra thick,red] (14, 2) -- (18,0);
	\draw [ultra thick,red] (18, 0) -- (22,2);
	\draw [dashed, thick,red] (22, 2) -- (18,4);
	\draw [dashed, thick,red] (22, 2) -- (22,6.5);
	\draw [fill=e,e] (18, 0) circle [radius=0.6];
	\draw [fill=e,e] (18, 4) circle [radius=0.6];
	\draw [fill=e,e] (14, 6.5) circle [radius=0.6];
	\draw [fill=e,e] (22, 6.5) circle [radius=0.6];
	\draw [fill=v,v] (14, 2) circle [radius=0.6];
	\draw [fill=v,v] (22, 2) circle [radius=0.6];
	\draw [fill=v,v] (18, 8.5) circle [radius=0.6];
	
	\draw [fill] (12.5,1.5) circle [radius=.3];
	\draw [->] (12.5,1.5) -- (13.6, 0.9);
	
	\node at (14, 6.5) {{\tiny{$0$}}};
	\node at (22, 6.5) {{\tiny{$0$}}};
	\node at (18, 4) {{\tiny{$1$}}};
	\node at (18, 0) {{\tiny{$1$}}};
	\end{scope}
	
	\begin{scope}[shift={(0,0)}]
	\draw [dashed, thick,red] (18, 8.5) -- (14,6.5);
	\draw [ultra thick,red] (14, 6.5) -- (14,2);
	\draw [ultra thick,red] (18, 8.5) -- (22,6.5);
	\draw [dashed, thick,red] (18, 8.5) -- (18,4);
	\draw [ultra thick,red] (14, 2) -- (18,4);
	\draw [ultra thick,red] (14, 2) -- (18,0);
	\draw [ultra thick,red] (18, 0) -- (22,2);
	\draw [dashed, thick,red] (22, 2) -- (18,4);
	\draw [ultra thick,red] (22, 2) -- (22,6.5);
	\draw [fill=e,e] (18, 0) circle [radius=0.6];
	\draw [fill=e,e] (18, 4) circle [radius=0.6];
	\draw [fill=e,e] (14, 6.5) circle [radius=0.6];
	\draw [fill=e,e] (22, 6.5) circle [radius=0.6];
	\draw [fill=v,v] (14, 2) circle [radius=0.6];
	\draw [fill=v,v] (22, 2) circle [radius=0.6];
	\draw [fill=v,v] (18, 8.5) circle [radius=0.6];
	
	\node at (14, 6.5) {{\tiny{$0$}}};
	\node at (22, 6.5) {{\tiny{$1$}}};
	\node at (18, 4) {{\tiny{$0$}}};
	\node at (18, 0) {{\tiny{$1$}}};
	\end{scope}
	
	\begin{scope}[shift={(14, 0)}]
	\draw [dashed, thick,red] (18, 8.5) -- (14,6.5);
	\draw [ultra thick,red] (14, 6.5) -- (14,2);
	\draw [ultra thick,red] (18, 8.5) -- (22,6.5);
	\draw [dashed, thick,red] (18, 8.5) -- (18,4);
	\draw [ultra thick,red] (14, 2) -- (18,4);
	\draw [dashed, thick,red] (14, 2) -- (18,0);
	\draw [ultra thick,red] (18, 0) -- (22,2);
	\draw [ultra thick,red] (22, 2) -- (18,4);
	\draw [ultra thick,red] (22, 2) -- (22,6.5);
	\draw [fill=e,e] (18, 0) circle [radius=0.6];
	\draw [fill=e,e] (18, 4) circle [radius=0.6];
	\draw [fill=e,e] (14, 6.5) circle [radius=0.6];
	\draw [fill=e,e] (22, 6.5) circle [radius=0.6];
	\draw [fill=v,v] (14, 2) circle [radius=0.6];
	\draw [fill=v,v] (22, 2) circle [radius=0.6];
	\draw [fill=v,v] (18, 8.5) circle [radius=0.6];
	
	\node at (14, 6.5) {{\tiny{$0$}}};
	\node at (22, 6.5) {{\tiny{$1$}}};
	\node at (18, 4) {{\tiny{$1$}}};
	\node at (18, 0) {{\tiny{$0$}}};
	\end{scope}
	
	\begin{scope}[shift={(28, 0)}]
	\draw [ultra thick,red] (18, 8.5) -- (14,6.5);
	\draw [ultra thick,red] (14, 6.5) -- (14,2);
	\draw [ultra thick,red] (18, 8.5) -- (22,6.5);
	\draw [ultra thick,red] (18, 8.5) -- (18,4);
	\draw [dashed, thick,red] (14, 2) -- (18,4);
	\draw [ultra thick,red] (14, 2) -- (18,0);
	\draw [ultra thick,red] (18, 0) -- (22,2);
	\draw [dashed, thick,red] (22, 2) -- (18,4);
	\draw [dashed, thick,red] (22, 2) -- (22,6.5);
	\draw [fill=e,e] (18, 0) circle [radius=0.6];
	\draw [fill=e,e] (18, 4) circle [radius=0.6];
	\draw [fill=e,e] (14, 6.5) circle [radius=0.6];
	\draw [fill=e,e] (22, 6.5) circle [radius=0.6];
	\draw [fill=v,v] (14, 2) circle [radius=0.6];
	\draw [fill=v,v] (22, 2) circle [radius=0.6];
	\draw [fill=v,v] (18, 8.5) circle [radius=0.6];
	
	\node at (14, 6.5) {{\tiny{$1$}}};
	\node at (22, 6.5) {{\tiny{$0$}}};
	\node at (18, 4) {{\tiny{$0$}}};
	\node at (18, 0) {{\tiny{$1$}}};
	\end{scope}
	
	\begin{scope}[shift={(-14, -14)}]
	\draw [ultra thick,red] (18, 8.5) -- (14,6.5);
	\draw [ultra thick,red] (14, 6.5) -- (14,2);
	\draw [ultra thick,red] (18, 8.5) -- (22,6.5);
	\draw [ultra thick,red] (18, 8.5) -- (18,4);
	\draw [dashed, thick,red] (14, 2) -- (18,4);
	\draw [dashed, thick,red] (14, 2) -- (18,0);
	\draw [ultra thick,red] (18, 0) -- (22,2);
	\draw [ultra thick,red] (22, 2) -- (18,4);
	\draw [dashed, thick,red] (22, 2) -- (22,6.5);
	\draw [fill=e,e] (18, 0) circle [radius=0.6];
	\draw [fill=e,e] (18, 4) circle [radius=0.6];
	\draw [fill=e,e] (14, 6.5) circle [radius=0.6];
	\draw [fill=e,e] (22, 6.5) circle [radius=0.6];
	\draw [fill=v,v] (14, 2) circle [radius=0.6];
	\draw [fill=v,v] (22, 2) circle [radius=0.6];
	\draw [fill=v,v] (18, 8.5) circle [radius=0.6];
	
	\node at (14, 6.5) {{\tiny{$1$}}};
	\node at (22, 6.5) {{\tiny{$0$}}};
	\node at (18, 4) {{\tiny{$1$}}};
	\node at (18, 0) {{\tiny{$0$}}};
	\end{scope}
	
	\begin{scope}[shift={(0, -14)}]
	\draw [ultra thick,red] (18, 8.5) -- (14,6.5);
	\draw [ultra thick,red] (14, 6.5) -- (14,2);
	\draw [ultra thick,red] (18, 8.5) -- (22, 6.5);
	\draw [dashed, thick,red] (18, 8.5) -- (18, 4);
	\draw [dashed, thick,red] (14, 2) -- (18, 4);
	\draw [dashed, thick,red] (14, 2) -- (18, 0);
	\draw [ultra thick,red] (18, 0) -- (22,2);
	\draw [ultra thick,red] (22, 2) -- (18,4);
	\draw [ultra thick,red] (22, 2) -- (22,6.5);
	\draw [fill=e,e] (18, 0) circle [radius=0.6];
	\draw [fill=e,e] (18, 4) circle [radius=0.6];
	\draw [fill=e,e] (14, 6.5) circle [radius=0.6];
	\draw [fill=e,e] (22, 6.5) circle [radius=0.6];
	\draw [fill=v,v] (14, 2) circle [radius=0.6];
	\draw [fill=v,v] (22, 2) circle [radius=0.6];
	\draw [fill=v,v] (18, 8.5) circle [radius=0.6];
	
	\node at (14, 6.5) {{\tiny{$1$}}};
	\node at (22, 6.5) {{\tiny{$1$}}};
	\node at (18, 4) {{\tiny{$0$}}};
	\node at (18, 0) {{\tiny{$0$}}};
	\end{scope}
	
	\begin{scope}[shift={(14, -14)}]
	\draw [dashed, thick,red] (18, 8.5) -- (14,6.5);
	\draw [ultra thick,red] (14, 6.5) -- (14,2);
	\draw [ultra thick,red] (18, 8.5) -- (22,6.5);
	\draw [ultra thick,red] (18, 8.5) -- (18,4);
	\draw [ultra thick,red] (14, 2) -- (18,4);
	\draw [dashed, thick,red] (14, 2) -- (18,0);
	\draw [ultra thick,red] (18, 0) -- (22,2);
	\draw [ultra thick,red] (22, 2) -- (18,4);
	\draw [dashed, thick,red] (22, 2) -- (22,6.5);
	\draw [fill=e,e] (18, 0) circle [radius=0.6];
	\draw [fill=e,e] (18, 4) circle [radius=0.6];
	\draw [fill=e,e] (14, 6.5) circle [radius=0.6];
	\draw [fill=e,e] (22, 6.5) circle [radius=0.6];
	\draw [fill=v,v] (14, 2) circle [radius=0.6];
	\draw [fill=v,v] (22, 2) circle [radius=0.6];
	\draw [fill=v,v] (18, 8.5) circle [radius=0.6];
	
	\node at (14, 6.5) {{\tiny{$0$}}};
	\node at (22, 6.5) {{\tiny{$0$}}};
	\node at (18, 4) {{\tiny{$2$}}};
	\node at (18, 0) {{\tiny{$0$}}};
	\end{scope}
	
	\begin{scope}[shift={(28, -14)}]
	\draw [ultra thick,red] (18, 8.5) -- (14,6.5);
	\draw [ultra thick,red] (14, 6.5) -- (14,2);
	\draw [ultra thick,red] (18, 8.5) -- (22,6.5);
	\draw [ultra thick,red] (18, 8.5) -- (18,4);
	\draw [dashed, thick,red] (14, 2) -- (18,4);
	\draw [ultra thick,red] (14, 2) -- (18,0);
	\draw [dashed, thick,red] (18, 0) -- (22,2);
	\draw [ultra thick,red] (22, 2) -- (18,4);
	\draw [dashed, thick,red] (22, 2) -- (22,6.5);
	\draw [fill=e,e] (18, 0) circle [radius=0.6];
	\draw [fill=e,e] (18, 4) circle [radius=0.6];
	\draw [fill=e,e] (14, 6.5) circle [radius=0.6];
	\draw [fill=e,e] (22, 6.5) circle [radius=0.6];
	\draw [fill=v,v] (14, 2) circle [radius=0.6];
	\draw [fill=v,v] (22, 2) circle [radius=0.6];
	\draw [fill=v,v] (18, 8.5) circle [radius=0.6];
	
	\node at (14,2) {{\tiny{$1$}}};
	\node at (22,2) {{\tiny{$2$}}};
	\node at (18,8.5) {{\tiny{$0$}}};
	\end{scope}
	\end{tikzpicture}
	\caption{Jaeger trees and the Bernardi process. The ribbon structure is the positive orientation of the plane, while the base node is the lower left blue node, with base edge going to the right and downwards. The first seven panels show the seven $V$-cut Jaeger trees of this bipartite graph, with the realized hypertrees on $E$. The gray line on the first panel indicates how the Bernardi process proceeds. The last panel shows a spanning tree which is not a Jaeger tree.}
	\label{fig:Jaeger trees}
\end{figure}

\begin{defn}[V-cut and $E$-cut Jaeger trees, \cite{Hyper_Bernardi}]
A spanning tree $T$ of $H$ is called a $V$-cut (resp.\ $E$-cut) Jaeger tree, if in the tour of $T$, each edge $\varepsilon\notin T$ is cut through at its violet (resp. emerald) endpoint.
\end{defn}

See Figure \ref{fig:Jaeger trees} for examples. 

\begin{thm}\cite[Corollary 6.16]{Hyper_Bernardi}\label{thm:Jaeger_trees_represent_hypertrees}
Let $H$ be a bipartite ribbon graph with a fixed edge $b_0b_1$. For each hypertree $f\in B_{E}(H)$, there is exactly one $V$-cut Jaeger tree $T$ such that $f=f_E(T)$, and for each hypertree $f\in B_V(H)$, there is exactly one $V$-cut Jaeger tree $T$ such that $f=f_V(T)$. In particular, the $V$-cut Jaeger trees give a bijection between $B_{E}(H)$ and $B_{V}(H)$.
\end{thm}

We will denote this bijection by $\beta_{b_0,b_1}:B_E(H) \to B_V(H)$. The notation intentionally agrees with the notation for the Bernardi bijection; let us show that this bijection generalizes the Bernardi bijection between spannning trees and break divisors.
In \cite{Hyper_Bernardi}, a greedy algorithm (called the hypergraphical Bernardi process) is given for finding the unique representing $V$-cut Jaeger tree for a hypertree $f$ on $E$. The process 
starts from $b_0$ and traverses or removes each edge of $H$. The first edge to be examined is $b_0b_1$. If we arrive at an edge from the emerald direction, the edge needs to be traversed, and we examine the next edge according to the ribbon structure at the new current vertex. If we arrive at an edge from the violet direction, we remove it if the hypertree $f$ can be realized in the graph after the removal of the edge. In this case we continue with the smaller graph, and we take the next edge (according to the ribbon structure) at the current vertex. If the edge cannot be removed, we traverse it and continue with the next edge at the new current vertex.
The process ends when we would examine an edge for the second time from the same direction. See also \cite[Definition 4.1]{Hyper_Bernardi}. In \cite{Hyper_Bernardi} it is proved that at the end of the process, the graph of remaining edges is the unique $V$-cut Jaeger tree $T$ representing $f$ and the Bernardi process traverses the graph the same way as the tour of $T$. (For hyergraphs, this is a nontrivial result.)

One can easily check that the bijection of Baker and Wang corresponds to the hypergraphical Bernardi process on $B=\Bip G$, where they drop a chip on the violet end of each removed edge (see also \cite[Remark 4.4]{Hyper_Bernardi}). Thus, the break divisor they obtain is exactly the dual pair of the hypertree obtained from the Jaeger tree.

Jaeger trees have many more nice properties, for example they correspond to a shellable dissection of the root-polytope of the bipartite graph (see \cite[Section 7]{Hyper_Bernardi}). Moreover, for planar ribbon structures, this dissection is a triangulation.

\section{An embedding of the Sandpile group in $\mathcal{A}_W$}
\label{sec:embedding_the_sandpile_group}

\begin{thm}\label{thm:sandpile_group_in_A_W}
	The equivalence classes of $\mathcal{A}_W$ containing at least one element of the form $(x_V,\mathbf{0},\mathbf{0})$ form a group $G$ isomorphic to $\Pic(D_V)\cong \Pic^0(D_V)\times \mathbb{Z}$.
\end{thm}
\begin{proof}
	$G$ is is clearly a subgroup of $\mathcal{A}_W$.
	We can think of $G$ as a free Abelian group on the set $V$ factorized by some relation induced by $\approx_W$. Hence for proving that $G\cong \Pic^0(D_V)\times \mathbb{Z}$, it suffices to show that the equivalence relation induced by $\approx_W$ on the elements of type $(x_V,\mathbf{0},\mathbf{0})$ coincides with linear equivalence of chip configurations on $D_V$. We show this in the following two lemmas.
    
    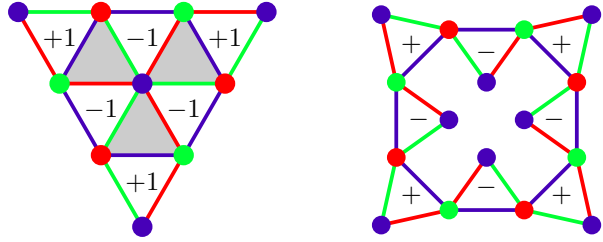
\begin{figure}
    \begin{center}
    \begin{tikzpicture}[-,>=stealth',auto,scale=1.1,thick]
    \draw[fill=light-gray] (0,0) -- (-0.5,0.866) -- (-1,0) -- cycle;
    \draw[fill=light-gray] (0,0) -- (0.5,0.866) -- (1,0) -- cycle;
    \draw[fill=light-gray] (0,0) -- (-0.5,-0.866) -- (0.5,-0.866) -- cycle;
	\path[every node/.style={font=\sffamily\small},color=r,line width=0.5mm]
	(0, 0) edge node {} (0.5, 0.866)
	(0, 0) edge node {} (0.5, -0.866)
	(0, 0) edge node {} (-1, 0)
	(-1.5, 0.866) edge node {} (-1, 0)
	(1.5, 0.866) edge node {} (0.5, 0.866)
	(0, -1.7321) edge node {} (0.5, -0.866);
	\path[every node/.style={font=\sffamily\small},color=v,line width=0.5mm]
	(-0.5, -0.866) edge node {} (0.5, -0.866)
	(-0.5, -0.866) edge node {} (-1, 0)
	(-0.5, 0.866) edge node {} (0.5, 0.866)
	(-0.5, 0.866) edge node {} (-1, 0)
	(1, 0) edge node {} (0.5, 0.866)
	(1, 0) edge node {} (0.5, -0.866);
	\path[every node/.style={font=\sffamily\small},color=e,line width=0.5mm]
	(0, 0) edge node {} (1, 0)
	(0, 0) edge node {} (-0.5, -0.866)
	(0, 0) edge node {} (-0.5, 0.866)
	(-1.5, 0.866) edge node {} (-0.5, 0.866)
	(1.5, 0.866) edge node {} (1, 0)
	(0, -1.7321) edge node {} (-0.5, -0.866);
    \tikzstyle{u}=[circle]
	\draw [fill=v,v] (0, 0) circle [radius=0.11];
    \draw [fill=v,v] (0, -1.7321) circle [radius=0.11];
    \draw [fill=v,v] (1.5, 0.866) circle [radius=0.11];
    \draw [fill=v,v] (-1.5, 0.866) circle [radius=0.11];
    \draw [fill=e,e] (0.5, 0.866) circle [radius=0.11];
    \draw [fill=e,e] (0.5, -0.866) circle [radius=0.11];
    \draw [fill=e,e] (-1, 0) circle [radius=0.11];
    \draw [fill=r,r] (1, 0) circle [radius=0.11];
    \draw [fill=r,r] (-0.5, -0.866) circle [radius=0.11];
    \draw [fill=r,r] (-0.5, 0.866) circle [radius=0.11];
    \node[u] at (-1,0.55) {$+1$};
    \node[u] at (0,0.55) {$-1$};
    \node[u] at (1,0.55) {$+1$};
    \node[u] at (-0.5,-0.31) {$-1$};
    \node[u] at (0.5,-0.31) {$-1$};
    \node[u] at (0,-1.2) {$+1$};
	\end{tikzpicture}
    \hspace{1cm}
    \begin{tikzpicture}[-,>=stealth',auto,scale=1,thick]
    \tikzstyle{u}=[circle]
	\path[every node/.style={font=\sffamily\small},color=r,line width=0.5mm]
	(0, 0.5) edge node {} (0.5, 1.2)
    (-0.5, 0) edge node {} (-1.2, 0.5)
    (1.4, 1.4) edge node {} (0.5, 1.2)
    (-1.4, -1.4) edge node {} (-0.5, -1.2)
    (0, -0.5) edge node {} (-0.5, -1.2)
    (0.5, 0) edge node {} (1.2, -0.5)
    (-1.4, 1.4) edge node {} (-1.2, 0.5)
    (1.4, -1.4) edge node {} (1.2, -0.5);
	\path[every node/.style={font=\sffamily\small},color=v,line width=0.5mm]
    (-0.5, 1.2) edge node {} (0.5, 1.2)
    (-1.2, -0.5) edge node {} (-1.2, 0.5)
    (-0.5, 1.2) edge node {} (-1.2, 0.5)
    (-1.2, -0.5) edge node {} (-0.5, -1.2)
    (-0.5, -1.2) edge node {} (0.5, -1.2)
    (1.2, -0.5) edge node {} (0.5, -1.2)
    (1.2, 0.5) edge node {} (0.5, 1.2)
    (1.2, 0.5) edge node {} (1.2, -0.5);
    \path[every node/.style={font=\sffamily\small},color=e,line width=0.5mm]
	(0, 0.5) edge node {} (-0.5, 1.2)
    (1.4, 1.4) edge node {} (1.2, 0.5)
    (-1.4, 1.4) edge node {} (-0.5, 1.2)
    (-1.2, -0.5) edge node {} (-0.5, 0)
    (-1.4, -1.4) edge node {} (-1.2, -0.5)
    (0, -0.5) edge node {} (0.5, -1.2)
    (0.5, 0) edge node {} (1.2, 0.5)
    (1.4, -1.4) edge node {} (0.5, -1.2);
	\draw [fill=v,v] (0, 0.5) circle [radius=0.11];
	\draw [fill=v,v] (1.4, -1.4) circle [radius=0.11];
	\draw [fill=v,v] (1.4, 1.4) circle [radius=0.11];
	\draw [fill=v,v] (0.5, 0) circle [radius=0.11];
	\draw [fill=v,v] (0, -0.5) circle [radius=0.11];
	\draw [fill=v,v] (-1.4, -1.4) circle [radius=0.11];
	\draw [fill=v,v] (-1.4, 1.4) circle [radius=0.11];
	\draw [fill=v,v] (-0.5, 0) circle [radius=0.11];
	\draw [fill=e,e] (0.5, 1.2) circle [radius=0.11];
	\draw [fill=e,e] (1.2, -0.5) circle [radius=0.11];
	\draw [fill=e,e] (-0.5, -1.2) circle [radius=0.11];
	\draw [fill=e,e] (-1.2, 0.5) circle [radius=0.11];
	\draw [fill=r,r] (1.2, 0.5) circle [radius=0.11];
	\draw [fill=r,r] (-1.2, -0.5) circle [radius=0.11];
	\draw [fill=r,r] (-0.5, 1.2) circle [radius=0.11];
	\draw [fill=r,r] (0.5, -1.2) circle [radius=0.11];
	\node[u]  at (0, 0.9) {$-$};
	\node[u]  at (0, -0.9) {$-$};
	\node[u]  at (0.9, 0) {$-$};
	\node[u]  at (-0.9, 0) {$-$};
	\node[u]  at (-1, 1) {+};
	\node[u]  at (-1, -1) {+};
	\node[u]  at (1, -1) {+};
	\node[u]  at (1, 1) {+};
	\end{tikzpicture}
    \end{center}
    \caption{Illustration for Lemmas \ref{lem:lin_ekv=>haromszog_ekv} and \ref{lem:haromszog_ekv=>lin_ekv}.}\label{fig:lemmas_illustration}
    \end{figure}

	\begin{lemma} \label{lem:lin_ekv=>haromszog_ekv}
		If $x_V\sim y_V$, then $(x_V,\mathbf{0},\mathbf{0})\approx_W (y_V,\mathbf{0},\mathbf{0})$.
	\end{lemma}
	\begin{proof}
		Since $\wapr$ is transitive, it is enough to show that if we get $y_V$ from $x_V$ by performing one firing on $D_V$, then $(x_V,\mathbf{0},\mathbf{0})\approx_W (y_V,\mathbf{0},\mathbf{0})$.
        
        Suppose that $y_V$ is obtained from $x_V$ by firing the node $v$ in $D_V$. We write $(y_V,\mathbf{0},\mathbf{0}) - (x_V,\mathbf{0},\mathbf{0})$ 
        as a a linear combination of white triangles, proving that $(x_V,\mathbf{0},\mathbf{0})\approx_W (y_V,\mathbf{0},\mathbf{0})$.
        
        First let us write $(y_V,\mathbf{0},\mathbf{0}) -(x_V,\mathbf{0},\mathbf{0})$ as a linear combination of black and white triangles. Take the black triangles incident to $v$ with coefficient $-1$, and the white triangles sharing a violet edge with one of the above mentioned black trianges with coefficient 1. If we look at $D_V$, then each black triangle corresponds to the tail of an edge of $D_V$, and each white triange corresponds to the head of an edge of $D_V$. The triangles we took correspond exactly to the out-edges of $v$. We claim that this linear combination gives $(y_V,\mathbf{0},\mathbf{0}) -(x_V,\mathbf{0},\mathbf{0})$.
        On violet nodes we clearly get $y_V-x_V$ since $v$ loses a chip for each out-edge, and we have a black triangle with coefficient $-1$ for each such edge, and a node gains one chip for each edge leading from $v$ to it, but for each such in-edge, we took a white triangle incident to the violet node with coefficient $1$.
        We need to prove that the sum remains zero on red and emerald nodes. To see this, notice that the triangles we took come in pairs of black and white triangles sharing a violet edge, with the black triangle having coefficient $-1$ and the white triangle having coefficient 1. 
        For any emerald node, some of the violet edges incident to it have a black triangle on one side with coefficient $-1$ and a white triangle on the other side with coefficient $1$, and some violet edges have triangles with coefficient zero on both sides. Hence altogether, the sum on any emerald node remains 0. We can say the same for red nodes.
        
        Now to have a linear combination of white triangles only, let us modify the above construction so that those white triangles that had coefficient 1 above should still have coefficient 1, but we take the white triangles incident to $v$ with coefficient $-1$ (and each black triangle with coefficient zero). This is now a linear combination of white triangles only, and we claim that this linear combination gives the same vector of $\mathbb{Z}^{V\cup E \cup R}$ as the previous one. Indeed, $v$ has the same number of incident black and white triangles, and any red or emerald neighbor of $v$ also has the same number of black and white triangles incident to $v$. As an illustration, see the left panel of Figure \ref{fig:lemmas_illustration}.
	\end{proof}
    
	\begin{lemma} \label{lem:haromszog_ekv=>lin_ekv}
		If $(x_V,\mathbf{0},\mathbf{0})\wapr (y_V,\mathbf{0},\mathbf{0})$, then $x_V\sim y_V$.
	\end{lemma}
	\begin{proof}
    Let $(y_V,\mathbf{0},\mathbf{0})-(x_V,\mathbf{0},\mathbf{0})=\sum_{1}^m a_i \mathbf{1}_{T_i}$, where $T_1,\dots, T_m$ are white triangles.
    First we claim that $\sum_1^k a_i=0$. As in both $(x_V,\mathbf{0},\mathbf{0})$ and $(y_V,\mathbf{0},\mathbf{0})$ the number of chips on each emerald node is 0, for each emerald node, the coefficients of the triangles incident to it have to sum to zero. As each triangle has exactly one emerald node, this implies that $\sum_1^k a_i=0$.
    
    We prove the Lemma by induction on $\sum |a_i|$.
    If $\sum|a_i|=0$, then $(x_V,\mathbf{0},\mathbf{0})= (y_V,\mathbf{0},\mathbf{0})$, that is $x_V= y_V$, so $x_V\sim y_V$.
    
    Now assume that $\sum|a_i|>0$. As $\sum a_i=0$, there is a triangle $T_{i_1}$ such that $a_{i_1}<0$. Let us denote the violet, emerald and red nodes of the triangle respectively by $v_1$, $e_1$, and $r_1$. As the coefficients of the triangles incident to $e_1$ sum to 0, there is a triangle $T_{i_2}$ incident to $e_1$ such that $a_{i_2}>0$. Let us call the nodes of $T_{i_2}$ respectively $v_2$, $e_2$ and $r_2$ (hence $e_1=e_2$). Now as the coefficients of the triangles incident to $r_2$ also sum to $0$, there is a triangle $T_{i_3}$ incident to $r_2$ such that $a_{i_3}<0$. We can continue this reasoning until we see a triangle for the second time. We conclude that (after possibly renumbering the triangles) there exists a sequence of triangles $T_{i_1}, \dots, T_{i_{2k}}$ such that $a_{i_{2j-1}}<0$ and $a_{i_{2j}}>0$ for each $1\leq j\leq k$ and $T_{i_{2j-1}}$ and $T_{i_{2j}}$ are incident at the emerald node $e_{2j-1}$ and $T_{i_{2j}}$ and $T_{i_{2j+1}}$ are incident at the red node $r_{2j}$ (where the indices are meant modulo $2k$).
    
	Take the cycle $(e_1,r_2,e_3,\dots, e_{2k-1},r_{2k})$ (which is a cycle in $G_V$). This cycle divides the plane into two components, where the triangles $T_{i_1},\dots T_{i_{2k-1}}$ fall into one component, while the triangles $T_{i_2},\dots T_{i_{2k}}$ fall into the other component (see the right panel of Figure \ref{fig:lemmas_illustration}). Let us call $U$ the set of violet nodes falling into the component containing $T_{i_1},\dots T_{i_{2k-1}}$. Take the chip configuration $x_V+ L_{D_V}\mathbf{1}_U$ on the digraph $D_V$, i.e., start from $x_V$ and then fire all vertices in $U$. Clearly $x_V\sim  x_V+ L_{D_V}\mathbf{1}_U$ and thus by Lemma \ref{lem:lin_ekv=>haromszog_ekv}, $(x_V,0,0)\wapr (x_V+ L_{D_V}\mathbf{1}_U,0,0)$. Now the transitivity of the white triangle equivalence implies that $(x_V+ L_{D_V}\mathbf{1}_U,0,0)\wapr (y_V,0,0)$. We claim that there exist coefficients $b_i$ such that $(y_V,0,0)-(x_V+ L_{D_V}\mathbf{1}_U,0,0)=\sum_{1}^m b_i T_i$ and $\sum |b_i|<\sum |a_i|$. This will imply the statement of the lemma by the inductive hypothesis and the transitivity of $\sim$.
    
    To see that there exist such coefficients $b_i$, take $c_{i_{2j-1}}=1$ and $c_{i_{2j}}=-1$ for each $1\leq j\leq k$, and $c_i=0$ if $i\neq i_j$ for any $1\leq j\leq k$. We claim that $\sum c_i T_i = (L_{D_V}\mathbf{1}_U,0,0)$. 
    To see this, notice that the violet edges $e_{2i-1}r_{2i}$ correspond to the edges of $D_V$ from $U$ to $V-U$, and the violet edges $e_{2i}r_{2i+1}$ correspond to the edges of $D_V$ from $V-U$ to $U$. Also, notice that if each node in $U$ is fired, then each node in $v\in V-U$ receives $d(U,v)$ chips. Moreover, each node $v\in U$ loses $d^+(v)$ chips and gains $d(U,v)$ chips. Hence altogether, a node $v\in U$ loses $d(V-U,v)$ chips.
    Now in $\sum c_i T_i$ any node $v\in V-U$ gains as many chips as many triangles of the form $T_{i_{2j-1}}$ are incident to it, which is exactly $d(U,v)$. Moreover, any node $v\in U$ loses as many chips as many triangles of the form $T_{i_{2j}}$ are incident to it, which is exactly $d(V-U,v)$. By this we proved that $\sum c_i T_i = (L_{D_V}\mathbf{1}_U,0,0)$.
    
    Now for $b_i=a_i-c_i$, we see that $\sum b_i T_i=\sum a_i T_i - \sum c_i T_i = [(y_V,0,0)-(x_V,0,0)]-(L_{D_V}\mathbf{1}_U,0,0) = (y_V,0,0)-(x_V+ L_{D_V}\mathbf{1}_U,0,0)$, and since $a_{i_{2j-1}}<0$ and $a_{i_{2j}}>0$ for each $1\leq j\leq k$, we have $\sum |b_i|<\sum |a_i|$.
    \end{proof}
	
This finishes the proof of Theorem \ref{thm:sandpile_group_in_A_W}.
\end{proof}

\begin{cor}
Those equivalence classes of $\mathcal{A}_W$ that contain at least one element of the form $(x_V,\mathbf{0},\mathbf{0})$ with $\deg(x_V)=0$ form a group isomorphic to $\Pic^0(D_V)$.
\end{cor}

As the roles of the three colors are completely symmetric for trinities, we obtain the same result for $D_E$ and $D_R$ as well.

\begin{cor}
	Those equivalence classes of $\mathcal{A}_W$ that contain at least one element of the form $(\mathbf{0}, x_E, \mathbf{0})$ with $\deg(x_E)=0$ form a group isomorphic to $\Pic^0(D_E)$, while those equivalence classes of $\mathcal{A}_W$ that contain at least one element of the form $(\mathbf{0}, \mathbf{0}, x_R)$ with $\deg(x_R)=0$ form a group isomorphic to $\Pic^0(D_R)$.
\end{cor}

The above embedding of the three sandpile groups immediately gives an isomorphism between them.

\begin{thm}\label{thm:isomorphism_between_sandpile_groups}
Let $\varphi_{V\to E}: \Pic^0(D_V) \to \Pic^0(D_E)$ be defined by
$\varphi_{V\to E}([x])=[y]$ where $(x,\mathbf{0},\mathbf{0})\wapr (\mathbf{0},-y,\mathbf{0})$. Then $\varphi_{V\to E}$ is well-defined and is an isomorphism between $\Pic^0(D_V)$ and $\Pic^0(D_E)$.
\end{thm}

\begin{proof} 
We show the well-definedness of $\varphi_{V\to E}$. We claim that if $(x,\mathbf{0},\mathbf{0})\wapr (\mathbf{0},-y,\mathbf{0})$ and $(x',\mathbf{0},\mathbf{0})\wapr (\mathbf{0},-y',\mathbf{0})$ for $x\sim x'$ (in $D_V$), then $y\sim y'$ in $D_E$. Indeed, Lemma \ref{lem:lin_ekv=>haromszog_ekv} implies $(x,\mathbf{0},\mathbf{0})\wapr (x',\mathbf{0},\mathbf{0})$, moreover, by the transitivity of $\wapr$, we have $(\mathbf{0},-y, \mathbf{0})\wapr (\mathbf{0},-y',\mathbf{0})$, hence also $(\mathbf{0},y, \mathbf{0})\wapr (\mathbf{0},y',\mathbf{0})$. By Lemma \ref{lem:haromszog_ekv=>lin_ekv} applied to $D_E$, this implies $y\sim y'$ (in $D_E$).

Now we show that for any $x\in \Pic^0(D_V)$ there exists $y\in \Pic^0(D_E)$ such that $(x,\mathbf{0},\mathbf{0})\wapr (\mathbf{0},-y,\mathbf{0})$.
If $x=\mathbf{0}$, then $y=\mathbf{0}$ is a good choice. If there exist a violet node $v$ with $x(v)>0$, then as the sum of chips in $x$ is zero, there exists another violet node $u$ with $x(u)<0$.
Choose a path in $G_E$ connecting $v$ with $u$ (there exists a path between them because of the connectedness of $G_E$). Now give weights $-1$ and $+1$ alternately to the white triangles incident with the path such that the triangle incident to $v$ gets coefficient $-1$ (see Figure \ref{fig:wel_def_of_isomorphism}). Then by parity, the triangle incident with $u$ has coefficient $+1$. Adding the characteristic vectors of these triangles with these weights to $(x,\mathbf{0},\mathbf{0})$, we decreased the sum of the absolute values of the chips on violet vertices, while the number of chips on each red node remained $0$. Continuing this way we reach a state with no chips on any violet or red node.

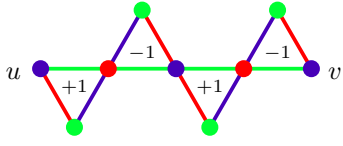
\begin{figure}
\begin{center}
	\begin{tikzpicture}[-,>=stealth',auto,scale=0.9,thick]
	\tikzstyle{u}=[circle]
	\node[u]  at (-1.9, 0.8) {$u$};
	\node[u]  at (2.85, 0.8) {$v$};
    \node[u]  at (-1, 0.6) {\scriptsize{$+1$}};
    \node[u]  at (0, 1.1) {\scriptsize{$-1$}};
    \node[u]  at (1, 0.6) {\scriptsize{$+1$}};
    \node[u]  at (2, 1.1) {\scriptsize{$-1$}};
	\path[every node/.style={font=\sffamily\small},color=r,line width=0.5mm]
	(-1.5, 0.866) edge node {} (-1, 0)
	(0.5, 0.866) edge node {} (0, 1.732)
	(0.5, 0.866) edge node {} (1, 0)
	(2.5, 0.866) edge node {} (2, 1.732);
	\path[every node/.style={font=\sffamily\small},color=v,line width=0.5mm]
	(-0.5, 0.866) edge node {} (-1, 0)
	(-0.5, 0.866) edge node {} (0, 1.732)
	(1.5, 0.866) edge node {} (1, 0)
	(1.5, 0.866) edge node {} (2, 1.732);
	\path[every node/.style={font=\sffamily\small},color=e,line width=0.5mm]
	(-1.5, 0.866) edge node {} (-0.5, 0.866)
	(-0.5, 0.866) edge node {} (0.5, 0.866)
	(0.5, 0.866) edge node {} (1.5, 0.866)
	(1.5, 0.866) edge node {} (2.5, 0.866);
	\draw [fill=v,v] (-1.5, 0.866) circle [radius=0.11];
	\draw [fill=r,r] (-0.5, 0.866) circle [radius=0.11];
	\draw [fill=v,v] (0.5, 0.866) circle [radius=0.11];
	\draw [fill=r,r] (1.5, 0.866) circle [radius=0.11];
	\draw [fill=v,v] (2.5, 0.866) circle [radius=0.11];
	\draw [fill=e,e] (-1, 0) circle [radius=0.11];
	\draw [fill=e,e] (0, 1.732) circle [radius=0.11];
	\draw [fill=e,e] (1, 0) circle [radius=0.11];
	\draw [fill=e,e] (2, 1.732) circle [radius=0.11];
	\end{tikzpicture}
	
\end{center}
\caption{Illustration for the proof of Theorem \ref{thm:isomorphism_between_sandpile_groups}.}
\label{fig:wel_def_of_isomorphism}
\end{figure}

We have shown that $\varphi_{V\to E}$ is well defined. Interchanging the roles of $V$ and $E$, the above two claims tell us that $\varphi_{V\to E}$ is injective and surjective.

It also follows immediately that $\varphi_{V\to E}$ is a homomorphism. Indeed, if we assume that $\varphi_{V\to E}([x_1])=[y_1]$, $\varphi_{V\to E}([x_2])=[y_2]$, and $\varphi_{V\to E}([x_1]+[x_2])=[y_0]$, then
\begin{align*}
(\mathbf{0},-y_0, \mathbf{0})\wapr (x_1 + x_2,\mathbf{0},\mathbf{0})\wapr (x_1,\mathbf{0},\mathbf{0}) + (x_2,\mathbf{0},\mathbf{0})\wapr \\
(\mathbf{0},-y_1,\mathbf{0}) + (\mathbf{0},-y_2,\mathbf{0})\wapr  (\mathbf{0},-(y_1+y_2),\mathbf{0}).
\end{align*}
Since $\varphi_{V\to E}$ is well defined, it follows that $-y_0 \sim -(y_1+y_2)$, and hence $\varphi_{V\to E}([x_1+x_2]) = \varphi_{V\to E}([x_1])+\varphi_{V\to E}([x_2])$.
\end{proof}

We can also define $\psi_{V\to E}: \Pic^0(D_V) \to \Pic^0(D_E)$ with $\psi([x])=[y]$ such that $(x,\mathbf{0}, \mathbf{0}) \wapr (\mathbf{0}, y, \mathbf{0})$. 
With a completely analogous proof, one can show that $\psi_{V\to E}$ is also an isomorphism between $\Pic^0(D_V)$ and $\Pic^0(D_E)$.

We claim that for planar undirected graphs, $\varphi_{V\to R}:\Pic^0(D_V) \to \Pic^0(D_R)$ agrees with the canonical isomorphism $i$ explained in Subsection \ref{ss:sandpile_and_Bernardi}. Note that the definition of $i$ was not canonical (it used (though was independent of) a choice of orientation of the graph), while the definition of $\varphi_{V\to R}$ is canonical.

\begin{prop} \label{prop:isomorphism_agree}
For planar undirected graphs, $\varphi_{V\to R}$ agrees with $i$.
\end{prop}

\begin{proof}
	Let $G$ be a planar undirected graph, and take the corresponding trinity.
We need to show that for an arbitrary orientation of $G$, and any $\{a_{\overrightarrow{e}}: e\in E\}$ we have $\varphi_{V\to R}([\sum_{e\in E} a_{\overrightarrow{e}} \delta_{\overrightarrow{e}}])=[\sum_{e\in E} a_{\overrightarrow{e}} \delta_{\overrightarrow{e}^*}]$. For this, it is enough to show that for any $e\in E$, we have $\varphi_{V\to R} ([\delta_{\overrightarrow{e}}]) = [\delta_{\overrightarrow{e}^*}]$. Let $v_h\in V$ be the head of $\overrightarrow{e}$ and let $v_t\in V$ be the tail of $\overrightarrow{e}$. Similarly, let $r_h\in R$ be the head of $\overrightarrow{e}^*$ and let $r_t\in R$ be the tail of $\overrightarrow{e}^*$. Let us also denote the emerald node corresponding to $e$ by $e$. We need to show that
$(\mathbf{1}_{v_h}-\mathbf{1}_{v_t},\mathbf{0},\mathbf{0})\wapr (\mathbf{0},\mathbf{0},\mathbf{1}_{r_t}-\mathbf{1}_{r_h})$ or in other words, $(\mathbf{1}_{v_h}-\mathbf{1}_{v_t},\mathbf{0},\mathbf{1}_{r_h}-\mathbf{1}_{r_t})\wapr (\mathbf{0},\mathbf{0},\mathbf{0})$. 
But the relationship of $\overrightarrow{e}$ and $\overrightarrow{e}^*$ implies that $v_h, r_h, e$ are the vertices of a white triangle, and $v_t, r_t, e$ are vertices of another white triangle. Hence taking the first triangle with coefficient $1$ and the second triangle with coefficient $-1$ proves the statement.
\end{proof}

\section{The Bernardi action agrees with the 
sandpile action on $D_E$}
\label{sec:Bernardi_and_sandpile_actions}

In this section, we give a canonical defintion for the planar Bernardi action of Baker and Wang by showing that it agrees with the natural sandpile action of $\Pic^0(D_E)$ on $\Pic^{|V|-1}(D_E)$ via the natural isomorphism $\varphi_{V\to E}:\Pic^0(D_V) \to \Pic^0(D_E)$. We conclude that the Bernardi action is independent of the base point for balanced plane digraphs, and we also give a simple proof for the compatibility with planar duality.

Let us repeat the definition of the Bernardi action using hypertree terminology, and discuss how the definition works for balanced plane digraphs.

For a graph $G$, the sandpile group $\Pic^0(G)$ acts on $\Pic^{|E|-1}(G)$ by addition: For $x\in \Pic^0(G)$ and $f\in \Pic^{|E|-1}(G)$, we let $x\cdot f= x+f$. Since by Theorem \ref{thm:break_represent}, the hypertrees in $B_V(\bip G)$ give a system of representatives for $\Pic^{|E|-1}(G)$, we can think of this natural action as the action of $\Pic^0(G)$ on $B_V(\bip G)$: for $x\in \Pic^0(G)$ and $f\in B_V(\bip G)$, we have $x\cdot f = x \oplus f$, where by $x\oplus f$ we denote the unique hypertree in $B_V(\bip G)$ in the linear equivalence class of $x + [f]$, which exists by Theorem \ref{thm:break_represent}. We call this group action the \emph{sandpile action}. Using Theorem \ref{thm:planar_hypertrees_represent} we can analogously define the sandpile action for balanced plane digraphs. For such a digraph $D_V$, the group $Pic^0(D_V)$ acts on $B_V(G_R)$ in the following way: For $x\in Pic^0(D_V)$ and $f\in B_V(G_R)$ we have $x\cdot f = x \oplus f$, where $x \oplus f$ is the unique hypertree in $B_V(G_R)$ in the linear equivalence class of $x + f$. Here linear equivalence is meant for $D_V$. Such a unique hypertree exists by Theorem \ref{thm:planar_hypertrees_represent}.

The Bernardi action is defined by pulling the sandpile action of $\Pic^0(G)$ from $B_V(\bip G)$ to $B_E(\bip G)$ using a Bernardi bijection:
for $x\in \Pic^0(G)$ and a hypertree $f\in B_E(\Bip G)$, 
$x\cdot f=\beta^{-1}_{b_0,b_1}(x\oplus \beta_{b_0,b_1}(f))$. Once again we can define the Bernardi action for a balanced plane digraph $D_V$ in the analogous way, replacing $\bip G$ by $G_R$, and thus (for each choice of $b_0$ and $b_1$) obtain a group action of $\Pic^0(D_V)$ on $B_E(G_R)$. From now on, we will concentrate on the case of balanced plane digraphs.

Let us clarify the relationship of our definition to the original definition of Baker and Wang. Take an undirected graph $G$. By Proposition \ref{prop:hypertree_break_div_relationship}, $f\in B_V(\Bip G)$ is a hypertree if and only if $d_G-\mathbf{1}-f$ is a break divisor. Hence if $x\oplus f=g$, that is, $g$ is a hypertree with $x+f\sim g$, then $d_G-\mathbf{1}-g$ is a break divisor linearly equivalent to $d_G-\mathbf{1}-f-x$.
That is, the sandpile action of $x\in\Pic^0(G)$ on a hypertree in $B_V(\Bip G)$ agrees with the sandpile action of $-x$ on the corresponding break divisor. Hence for a spanning tree $T$, the image $x\cdot T$ in our definition corresponds to $-x\cdot T$ in the definition of Baker and Wang, but here $x\mapsto -x$ is obviously an automorphism of $\Pic^0(G)$.

The following theorem is the main technical result of the section.

\begin{thm} \label{thm:Bernardi_compaticle_w_sandpile_action}
For a balanced plane digraph with ribbon structure coming from the positive orientation of the plane, and any choice of $\{b_0, b_1\}$, the Bernardi bijection commutes with the sandpile actions. That is, the following diagram is commutative.
\begin{center}
\begin{tikzpicture}
  \matrix (m) [matrix of math nodes,row sep=3em,column sep=4em,minimum width=2em]
  {
     B_V(G_R) & B_V(G_R) \\
     B_E(G_R) & B_E(G_R) \\};
  \path[-stealth]
    (m-2-1) edge node [left] {$\beta_{b_0,b_1}$} (m-1-1)
    (m-1-1) edge node [below] {$x$} (m-1-2)
    (m-2-1.east|-m-2-2) edge node [above] {$\varphi_{V\to E}(x)$} (m-2-2)
    (m-2-2) edge node [right] {$\beta_{b_0,b_1}$} (m-1-2);
\end{tikzpicture}
\end{center}
\end{thm}

Let us first discuss the corollaries of Theorem \ref{thm:Bernardi_compaticle_w_sandpile_action}. First of all, we obtain a canonical (once the choice of positive orientation is fixed) definition for the Bernardi action on balanced plane digraphs:

\begin{cor}\label{cor:Bernardi_can_def}
	For a balanced plane digraph $D_V$, element $x\in \Pic^0(D_V)$, and $f\in B_E(G_R)$, we have $x\cdot f= \varphi_{V\to E}(x) \oplus f$.
\end{cor}

As the definition of the sandpile action and the isomorphism $\varphi_{V\to E}$ was canonical, we further obtain:

\begin{cor}
	For balanced plane digraphs, the Bernardi action is independent of the choice of $b_0$ and $b_1$.
\end{cor}

The canonical duality between spanning trees of planar dual graphs generalizes to trinities in the following way (see \cite[Theorem 8.3]{hiperTutte}): 
If $f$ is a hypertree of $G_R$ on $E$, then $f^*=d_{G_R}|_E - \mathbf{1} - f = d_{G_V}|_E - \mathbf{1} - f$ is a hypertree of $G_V$ on $E$ (where $d_{G_R}$ once again means the vector of degrees of $G_R$). These are called planar dual hypertrees. Note that if the trinity is obtained from a planar graph (and so emerald hypertrees of $G_R$ are exactly the characteristic vectors of spanning trees of $G$, and emerald hypertrees of $G_V$ are exactly the characteristic vectors of spanning trees of $G^*$), then two hypertrees are planar dual if and only if they are the characteristic vectors of planar dual spanning trees. Hence planar duality of hypertrees indeed generalizes planar duality of spanning trees.

For the sandpile action of $Pic^0(D_E)$ on $B_E(G_R)$, it is extremely simple to prove compatibility with planar duality.

\begin{thm}\label{thm:compatibility_of_the_natural_action}
	The sandpile action of $Pic^0(D_E)$ on the emerald hypertrees of $G_R$ and the sandpile action of $Pic^0(D_E)$ on the emerald hypertrees of $G_V$ are compatible with planar duality.
	In other words, for any $[x]\in \Pic^0(D_E)$ and $f\in B_E(G_R)$, we have
	$([x]\oplus f)^* = [-x] \oplus f^*$.
\end{thm}

\begin{proof}
	Let $[x]\in Pic^0(D_E)$ and $f\in B_E(G_R)$ be arbitrary and put $x \oplus f= h$. Then $x\sim h-f$. It is enough to show that $-x\sim h^*-f^*$, but this is easy to see: $h^*-f^*=(d_{G_R}|_E - h)-(d_{G_R}|_E - f) = f - h  \sim -x$.
\end{proof}

The compatibility of the Bernardi action with planar duality is an immediate corollary.

\begin{cor}\label{cor:compatibility}
	The Bernardi action for balanced plane digraphs is compatible with planar duality.
	In other words, for any $[x]\in Pic^0(D_V)$ and $f\in B_E(G_R)$, we have $([x] \cdot f)^* = \varphi_{V\to R}([x]) \cdot f^*$.
\end{cor}

\begin{proof}
	By Corollary \ref{cor:Bernardi_can_def}, 
	$$[x] \cdot f=\varphi_{V\to E}([x]) \oplus f,$$ 
	and 
	$$\varphi_{V\to R}([x]) \cdot f^*=\varphi_{R\to E}(\varphi_{V\to R}([x])) \oplus f^*.$$ 
	Since it is clear from the definition that $\varphi_{R\to E}\circ \varphi_{V\to R}= \psi_{V\to E}$, we see that $\varphi_{V\to R}([x]) \cdot f^*=\psi_{V\to E}([x]) \oplus f^*$. Now the statement follows from Theorem \ref{thm:compatibility_of_the_natural_action} and the fact that $-\varphi_{V\to E}([x])=\psi_{V\to E}([x])$.
\end{proof}

\begin{figure}[ht]
	\begin{center}
		\begin{tikzpicture}[-,>=stealth',auto,scale=0.4,
		thick]
		\draw[light-gray,fill=light-gray,rounded corners=10pt] (0, 6.2) -- (2.8, 6) -- (5, 3.6) -- (5, 7.5) -- (-1.5, 7.5) -- (-3, 4.7) -- cycle;
		\draw[light-gray,fill=light-gray,rounded corners=10pt] (0, 0.8) -- (2.8, 0.5) -- (5, 3.6) -- (5, -1) -- (-1.5, -1) -- (-3, 2.3) -- cycle;
		\draw[fill=light-gray] (-3, 4.7) -- (-1.5, 4.2) -- (-3, 2.3) -- cycle;
		\draw[fill=light-gray] (0, 3.6) -- (-1.5, 4.2) -- (0, 6.2) -- cycle;
		\draw[fill=light-gray] (0, 3.6) -- (-3, 2.3) -- (0,2.2) -- cycle;
		\draw[fill=light-gray] (0, 0.8) -- (3, 2.3) -- (0,2.2) -- cycle;
		\draw[fill=light-gray] (0, 3.6) -- (3, 2.3) -- (1.5, 4.2) -- cycle;
		\draw[fill=light-gray] (3, 4.7) -- (0, 6.2) -- (1.5, 4.2) -- cycle;
		\draw[fill=light-gray] (3, 4.7) -- (3, 2.3) -- (5,3.6) -- cycle;
		\tikzstyle{e}=[{circle,draw,fill,color=e}]
		\tikzstyle{v}=[{circle,draw,fill,color=v}]
		\tikzstyle{r}=[{circle,draw,fill,color=r}]
		\node[e] (0) at (0, 0.8) {};
		\node at (-0.4, 0.3) {$b_1$};
		\node[v] (1) at (3, 2.3) {};
		\node[v,label=left:$b_0$] (2) at (-3, 2.3) {};
		\node[e] (3) at (3, 4.7) {};
		\node[e] (4) at (-3, 4.7) {};
		\node[v] (5) at (0, 6.2) {};
		\node[e] (6) at (0, 3.6) {};
		\node[r,label=right: \ $s_0$] (7) at (5, 3.6) {};
		\node[r] (8) at (1.5, 4.2) {};
		\node[r] (9) at (-1.5, 4.2) {};
		\node[r] (10) at (0, 2.2) {};
		\node at (-0.7, 2.6) {$s_1$};
		\draw[-, line width=0.4mm, color=e] (8) to (1);
		\draw[-, line width=0.4mm, color=v] (8) to (3);
		\draw[-, line width=0.4mm, color=e] (8) to (5);
		\draw[-, line width=0.4mm, color=v] (8) to (6);
		\draw[-, line width=0.4mm, color=v] (9) to (6);
		\draw[-, line width=0.4mm, color=e] (9) to (5);
		\draw[-, line width=0.4mm, color=v] (9) to (4);
		\draw[-, line width=0.4mm, color=e] (9) to (2);
		\draw[-, line width=0.4mm, color=v] (10) to (6);
		\draw[-, line width=0.4mm, color=e] (10) to (2);
		\draw[-, line width=0.4mm, color=v] (10) to (0);
		\draw[-, line width=0.4mm, color=e] (10) to (1);
		\draw[-, line width=0.4mm, color=e] (7) to (1);
		\draw[-, line width=0.4mm, color=v] (3) to (7);
		\draw[-, line width=0.4mm, color=e,rounded corners=10pt] (7) -- (2.8, 6) -- (5);
		\draw[-, line width=0.4mm, color=v, rounded corners=10pt] (4) -- (-1.5, 7.5) -- (5, 7.5) -- (7);
		\draw[-, line width=0.4mm, color=v,rounded corners=10pt] (7) -- (2.8, 0.5) -- (0);
		\draw[-, line width=0.4mm, color=e, rounded corners=10pt] (2) -- (-1.5, -1) -- (5, -1) -- (7);
		\draw[-, dashed, line width=0.4mm, rounded corners=5pt] (10) -- (-1.6, 0.5) -- (-1, -0.5) -- (3.5, -0.5) -- (7);
		\path[every node/.style={font=\sffamily\small},color=red,line width=1mm]
		(0) edge node {} (2)
		(1) edge node {} (0)
		(2) edge node {} (4)
		(5) edge node {} (3)
		(6) edge node {} (2)
		(5) edge node {} (6);
		\path[every node/.style={font=\sffamily\small},color=red,line width=0.4mm]
		(4) edge node {} (5)
		(1) edge node {} (3)
		(1) edge node {} (6);
		\draw[-,line width=0.8mm,rounded corners=10pt] (7) -- (3.2, 6.9) -- (0, 6.9)  -|  (9);
		\draw[-,line width=0.8mm] (7) to[bend left=30] (8);
		\draw[-,line width=0.8mm] (8) to[bend left=30] (10);
		\end{tikzpicture}
	\end{center}
	\caption{A spanning tree of $G_R$ (thick red lines) and the corresponding spanning tree of $G_R^*$ (thick black lines).}
	\label{f:G_R^*}
\end{figure}
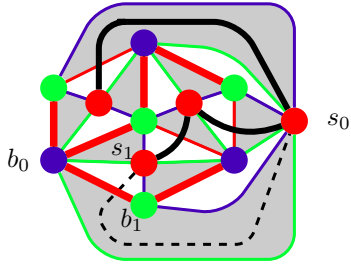

Now we prepare to prove Theorem \ref{thm:Bernardi_compaticle_w_sandpile_action}. First, we need to examine the relationship of the tour of a spanning tree and duality.  This was also examined in \cite[Section 8.2]{Bernardi_first}, but we repeat it in our language since we need special corollaries for bipartite graphs. Take the dual $G_R^*$ of $G_R$. Note that $G_R^*$ is the undirected graph we get by forgetting the orientation of the edges of $D_R$.
$G_R^*$ is also embedded in the sphere. Let us take the ribbon structure on it coming from the negative orientation of the sphere. For a spanning tree $T$ of $G_R$, let $T^*$ be the spanning tree of $G_R^*$ that consists of the edges not contained in $T$. First, suppose that $b_0$ is violet and $b_1$ is emerald. Then let $s_0s_1$ be the dual edge of $b_0b_1$ such that $s_0b_0b_1$ is a black triangle. (See Figure \ref{f:G_R^*}, where the spanning tree $T$ is drawn by thick red lines, $T^*$ is drawn by thick black lines, and $s_0s_1$ is drawn by dashed line. Note that we did not draw all the edges of $G_R^*$.) If $b_0$ is emerald and $b_1$ is violet, then let $s_0s_1$ be the dual edge of $b_0b_1$ such that $s_0b_0b_1$ is a white triangle. We claim that the tour of $T$ in $G_R$ with starting node $b_0$ and starting edge $b_0b_1$ using the positive orientation is ``the same'' as the tour of $T^*$ in $G_R^*$ with starting node $s_0$ and starting edge $s_0s_1$ using the negative orientation. By ``the same'', we understand that at any moment, if the current node in $G_R$ is $b$ and the current edge is $bb'$, then the current node in $G_R^*$ is $s$ and the current edge is $ss'$ where $ss'$ is the dual edge of $bb'$ where we get $ss'$ by turning $bb'$ in the positive direction. In other words, $ss'$ connects the red nodes of the two triangles sharing the edge $bb'$, and the triangle $sbb'$ is black if and only if $b$ is violet.
To see this, note that it is true at the beginning, and it stays true after a step of the tours.

Cutting through an edge of $G_R$ corresponds to traversing the corresponding edge of $G_R^*$ and vice versa. More precisely, cutting through an edge at the violet endpoint corresponds to traversing the corresponding edge of $G_R^*$ compatibly with the orientation in $D_R$ and cutting through an edge at the emerald endpoint corresponds to traversing the corresponding edge of $G_R^*$ opposite to the orientation in $D_R$. If $T$ is a $V$-cut Jaeger tree, then the edges of $G_R$ are always cut through at their violet endpoint. Hence in the tour of $T^*$, each edge is traversed first in the black-to-white direction, i.e., compatibly with the orientation in $D_R$. And vice versa, if in the tour of $T^*$, each edge is traversed first compatibly with the orientation in $D_R$, then in the tour of $T$, each edge is cut through at its violet endpoint.
This implies the following property (which was also pointed out in \cite{Hyper_Bernardi}).

\begin{prop}\label{prop:Jaeger--arborescence}
	Let $T$ be a spanning tree of $G_R$ in a trinity, and let $s_0$ be chosen such that if $b_0$ is violet, then $b_0b_1s_0$ is a black triangle, and if $b_0$ is emerald, then $b_0b_1s_0$ is a white triangle. Then
	$T$ is a $V$-cut Jaeger tree of $G_R$ with base point $b_0$ and base edge $b_0b_1$ if and only if $T^*$ is an arborescence of $D_R$ rooted at $s_0$. \hfill\qedsymbol
\end{prop}

\begin{remark}
	Proposition \ref{prop:Jaeger--arborescence} gives an alternative explanation for the observation of Yuen \cite[Section 5.1]{Yuen17} that the Bernardi bijection for a planar graph depends only on the face to the right of the starting edge (in our language: on $s_0$ where $b_0b_1s_0$ is black and $b_0$ is violet). Indeed, by Proposition \ref{prop:Jaeger--arborescence} for a given $s_0$ we get the same Jaeger trees independent of the actual $b_0$ and $b_1$, hence the Bernardi bijection is also the same. 
\end{remark}

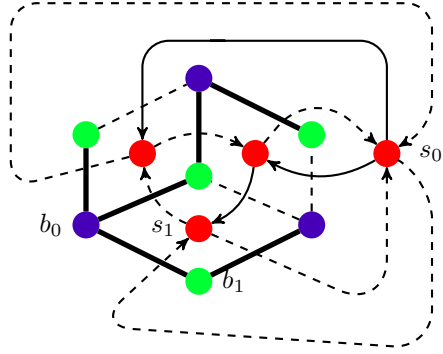
\begin{figure}
\begin{center}
    \begin{tikzpicture}[-,>=stealth',auto,scale=0.5,
	thick]
	\tikzstyle{e}=[{circle,draw,fill,color=e,font=\sffamily\small}]
	\tikzstyle{v}=[{circle,draw,fill,color=v,font=\sffamily\small}]
	\tikzstyle{r}=[{circle,draw,fill,color=r,font=\sffamily\small}]
	\node[e,label=right:$b_1$] (0) at (0, 0.8) {};
	\node[v] (1) at (3, 2.3) {};
	\node[v,label=left:$b_0$] (2) at (-3, 2.3) {};
	\node[e] (3) at (3, 4.7) {};
	\node[e] (4) at (-3, 4.7) {};
	\node[v] (5) at (0, 6.2) {};
	\node[e] (6) at (0, 3.6) {};
    \node[r,label=right: \ $s_0$] (7) at (5, 4.2) {};
    \node[r] (8) at (1.5, 4.2) {};
    \node[r] (9) at (-1.5, 4.2) {};
    \node[r, label=left:$s_1$] (10) at (0, 2.2) {};
    \draw[->,rounded corners=10pt] (7) |- (1, 7.2) -- (0, 7.2)  -|  (9);
    \draw[->] (7) to[bend left=30] (8);
    \draw[->] (8) to[bend left=30] (10);
    \draw[->, dashed, rounded corners=20pt] (8) -- (3, 6) --  (7);
    \draw[->, dashed] (10) to[bend left=30] (9);
    \draw[->, dashed] (9) to[bend left=30] (8);
    \draw[->, dashed, rounded corners=10pt] (9) -- (-5, 3.3) |- (0, 8.2) -| (6.2, 4.9) -- (7);
    \draw[->, dashed, rounded corners=20pt] (10) -- (5, 0) --  (7);
    \draw[->, dashed, rounded corners=15pt] (7) -- (6.2, 3.5) -- (6.2, -1) -- (-2.8, -0.4) --  (10);
	\path[every node/.style={font=\sffamily\small},dashed]
	(4) edge node {} (5)
    (1) edge node {} (3)
    (1) edge node {} (6);
	\path[every node/.style={font=\sffamily\small}, line width=0.7mm]
	(0) edge node {} (2)
    (1) edge node {} (0)
    (2) edge node {} (4)
	(5) edge node {} (3)
    (6) edge node {} (2)
	(5) edge node {} (6);
	\end{tikzpicture}
\end{center}
\caption{An illustration for Proposition \ref{prop:Jaeger--arborescence}.}
\label{f:Jaeger_tours_and_duality}
\end{figure}
\begin{proof}[Proof of Theorem \ref{thm:Bernardi_compaticle_w_sandpile_action}]
Let us take two arbitrary hypertrees $h$ and $h'$ of $G_R$ on $V$.
The statement of the theorem is equivalent to the fact that $\varphi_{V\to E}([h'-h])=[\beta_{b_0,b_1}^{-1}(h')-\beta_{b_0,b_1}^{-1}(h)]$.
By Theorem \ref{thm:Jaeger_trees_represent_hypertrees}, there exist $V$-cut Jaeger trees $T$ and $T'$ (with the fixed starting vertex and edge) such that $h=f_V(T)$ and $h'=f_V(T')$. Again by Theorem \ref{thm:Jaeger_trees_represent_hypertrees}, in this case their preimages at the Bernardi-bijection are $\beta^{-1}_{b_0,b_1}(h)=f_E(T)$ and $\beta^{-1}_{b_0,b_1}(h')=f_E(T')$. Then we have to show that for $f_V(T')-f_V(T)$ taken as a chip configuration on $D_V$, it follows that $$\varphi_{V\to E}([f_V(T')-f_V(T)])= [f_E(T')-f_E(T)],$$ that is, $(f_V(T')-f_V(T),0,0)\wapr (0,f_E(T)-f_E(T'),0)$ or equivalently, $$(f_V(T')-f_V(T),f_E(T')-f_E(T),0)\wapr (0,0,0).$$

Let us suppose that $T'=T-\{v_1e_1, \dots , v_ke_k\} + \{v'_1e'_1, \dots, v'_ke'_k\}$. Then $(f_V(T')-f_V(T),f_E(T')-f_E(T),0) = (\mathbf{1}_{v'_1}+ \dots + \mathbf{1}_{v'_k}-(\mathbf{1}_{v_1}+ \dots + \mathbf{1}_{v_k}),\mathbf{1}_{e'_1}+ \dots + \mathbf{1}_{e'_k} - (\mathbf{1}_{e_1}+ \dots + \mathbf{1}_{e_k}),0)$.

Let us examine the arborescences $A$ and $A'$ dual to $T$ and $T'$, respectively. That is, $A$ consists of the edges of $D_R$ such that the corresponding edge of $G_R$ is not in $T$, and similarly for $A'$ and $T'$. Let us denote the directed edge corresponding to an edge $\varepsilon$ of $G_R$ by $\varepsilon^*$. We get $A'$ by adding $(v_ie_i)^*$ to $A$ for $i=1,\dots k$, then removing $(v'_ie'_i)^*$ again for every $i=1,\dots k$. Let $r_i$ be the red node that is the head of $(v_ie_i)^*$. In an arborescence of $D_R$, every red node except for $r_0$ has indegree 1, and $r_0$ has indegree 0. This implies that for each $i=1,\dots k$ we have $r_i\neq r_0$, since $r_i$ has indegree at least one in $A'$. It also follows that $r_1, \dots r_k$ are all different nodes, otherwise $A'$ would have a node with indegree larger than one. As $r_i\neq r_0$ for each $i$, we see that $r_i$ also has indegree 1 in $A$. As $(v_ie_i)^*\notin A$ (since $v_ie_i\in T$), this means that by adding $(v_ie_i)^*$ to $A$ for every $i$, the indegree of $r_i$ increases to $2$. Hence for some $j$ the edge $(v'_je'_j)^*$ also has $r_i$ as its head. By relabeling, we can suppose that for each $i$, the edges $(v_ie_i)^*$ and $(v'_ie'_i)^*$ both have head $r_i$. This means that $r_iv_ie_i$ and $r_iv'_ie'_i$ are both white triangles of the trinity for each $i$. Now for each $i$, we can take $r_iv'_ie'_i$ with coefficient 1 and $r_iv_ie_i$ with coefficient $-1$ and obtain $(\mathbf{1}_{v'_1}+ \dots + \mathbf{1}_{v'_k}-(\mathbf{1}_{v_1}+ \dots + \mathbf{1}_{v_k}),\mathbf{1}_{e'_1}+ \dots + \mathbf{1}_{e'_k} - (\mathbf{1}_{e_1}+ \dots + \mathbf{1}_{e_k}),0)$
 as an integer linear combination of white triangles, finishing the proof.

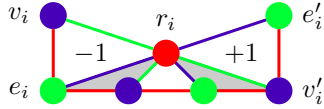
\begin{center}
	\begin{tikzpicture}[-,>=stealth',auto,scale=0.5,
	thick]
    \draw[fill=light-gray] (4,0) -- (6,0) -- (3,1) -- cycle;
	\draw[fill=light-gray] (0,0) -- (2,0) -- (3,1) -- cycle;
	\tikzstyle{r}=[color=r,circle,fill,draw,font=\sffamily\small]
	\tikzstyle{v}=[color=v,circle,fill,draw]
	\tikzstyle{e}=[color=e,circle,fill,draw]
	\tikzstyle{u}=[circle]
	\node[v,label=left:{$v_i$}] (02) at (0, 2) {};
	\node[e,label=right:{$e'_i$}] (62) at (6, 2) {};
	\node[e,label=left:{$e_i$}] (00) at (0, 0) {};
	\node[v] (20) at (2, 0) {};
	\node[e] (40) at (4, 0) {};
	\node[v,label=right:{$v'_i$}] (60) at (6, 0) {};
	\node[r, label=above:$r_i$] (11) at (3, 1) {};
	\node[u] at (1,1) {$-1$};
	\node[u] at (5,1) {$+1$};
	\path[every node/.style={font=\sffamily\small},color=r,line width=0.4mm]
	(00) edge node {} (20)
	(20) edge node {} (40)
	(40) edge node {} (60)
	(60) edge node {} (62)
	(02) edge node {} (00);
	\path[every node/.style={font=\sffamily\small},color=v,line width=0.4mm]
	(00) edge node {} (11)
	(40) edge node {} (11)
	(62) edge node {} (11);
	\path[every node/.style={font=\sffamily\small},color=e,line width=0.4mm]
	(02) edge node {} (11)
	(60) edge node {} (11)
	(20) edge node {} (11);
	\end{tikzpicture}
\end{center}

\end{proof}

\section{Proof of Theorem \ref{thm:planar_hypertrees_represent}.}\label{sec:representation}

We will need a characterization of hypertrees from \cite{hiperTutte}. For a set $S\subseteq V$, let us denote by $\Gamma_{G_R}(S)$ the set of nodes from $E$ that are connected to any node of $S$ by an edge of $G_R$.

\begin{thm}\cite[Theorem 3.4]{hiperTutte} \label{thm:hypertree_char}
$f$ is a hypertree of $G_R$ on $V$ if and only if
\begin{enumerate}
\item[(i)] $f(S) \leq |\Gamma_{G_R}(S)|-1$ for any $S\subseteq V$,
\item[(ii)] $f(V) = |E|-1$.
\end{enumerate} \hfill\qedsymbol
\end{thm}
We note that \cite[Theorem 3.4]{hiperTutte} also includes the condition that $f(v)\geq 0$ for each $v\in V$, but this follows from $(i)$ and $(ii)$.

\begin{proof}[Proof of Theorem \ref{thm:planar_hypertrees_represent}]
First we show that in any linear equivalence class of $\Pic(D_V)$ of degree $|E|-1$, there is at most one hypertree of $G_R$ on $V$. Suppose for a contradiction that there exist two hypertrees $f, f' \in B_{G_R}(V)$ such that $f\sim f'$ in $D_V$. This means that there exist $z\in \mathbb{Z}^V$ such that $f'=f + L_{D_V}z$. Since $L_{D_V}\mathbf{1}=\mathbf{0}$, we can suppose that $z$ only has nonnegative elements, and it has a zero coordinate.

Let $S=\{v\in V: z(v)=0\}$. Then $f'(S)\geq f(S) + d_{D_V}(V-S, S)$. Indeed, we can get from $f$ to $f'$ by firing each node $v\in V$ exactly $z(v)$ times, in which case nodes of $S$ do not fire, while each node of $V-S$ fires at least once. Thus $S$ does not lose any chips, and it receives at least one chip through each edge leading from $V-S$ to $S$.  
On the other hand,
$$f(S)\geq |E|-|\Gamma_{G_R}(V-S)|$$
using $f(S)=f(V)-f(V-S)$ and Theorem \ref{thm:hypertree_char}.
Hence $$f'(S)\geq f(S) + d_{D_V}(V-S,S)\geq |E|-|\Gamma_{G_R}(V-S)|+ d_{D_V}(V-S,S).$$

The number $d_{D_V}(V-S,S)$ counts the directed edges leading from $V-S$ to $S$, which is the number of the edges of $G_V$ (i.e., violet edges) such that the black triangle incident to them has a violet node from $V-S$ and the white triangle incident to them has a violet node from $S$. Notice that for each emerald node $e$ that has neighbors both from $S$ and from $V-S$, there is at least one violet edge incident to $e$ with the above property. Indeed, if we look at the violet neighbors of $e$ in a positive cyclic order, there must be a time where after a neighbor from $V-S$, we see a neighbor from $S$. The violet edge incident to $e$ separating the triangles of these two neighbors will be appropriate. Hence $d_{D_V}(V-S,S)$ is at least the number of emerald nodes that have neighbors from both $S$ and $V-S$ in $G_R$. 
Now $|E|-|\Gamma_{G_R}(V-S)|+ d_{D_V}(V-S,S)\geq|\Gamma_{G_R}(S)|$, since from the number of emerald nodes we subtracted those that are neighbors of $V-S$ (in $G_R$), but then added back at least the number of those nodes that are also neighbors of $S$. This means $f'(S)\geq |\Gamma_{G_R}(S)|$ which contradicts the fact that $f'$ is a hypertree. With this we have proved that any linear equivalence class of $\Pic(D_V)$ of degree $|E|-1$ contains at most one hypertree.

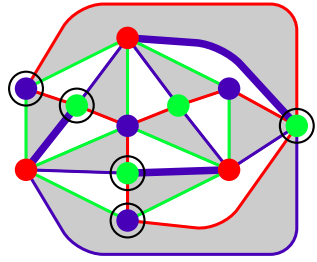
\begin{figure}
	\begin{center}
		\begin{tikzpicture}[-,>=stealth',auto,scale=0.45,
		thick]
		\draw[fill=light-gray,color=light-gray] (0, 6.2) -- (2.8, 6) -- (5, 3.6) -- (4.6, 7) -- (-1, 7) -- (-3, 4.7) -- cycle;
		\draw[color=light-gray,fill=light-gray,rounded corners=10pt] (0, 6.2) -- (2.8, 6) -- (5, 3.6) -- (5, 7.2) -- (-1.5, 7.2) -- (-3, 4.7) -- cycle;
		\draw[color=light-gray,fill=light-gray,rounded corners=10pt] (0, 0.8) -- (2.8, 0.5) -- (5, 3.6) -- (5, -0.2) -- (-1.5, -0.2) -- (-3, 2.3) -- cycle;
		\draw[color=light-gray,fill=light-gray] (0, 0.8) -- (2.8, 0.5) -- (5, 3.6) -- (4.7, 0) -- (-1, 0) -- (-3, 2.3) -- cycle;
		\draw[fill=light-gray] (-3, 4.7) -- (-1.5, 4.2) -- (-3, 2.3) -- cycle;
		\draw[fill=light-gray] (0, 3.6) -- (-1.5, 4.2) -- (0, 6.2) -- cycle;
		\draw[fill=light-gray] (0, 3.6) -- (-3, 2.3) -- (0,2.2) -- cycle;
		\draw[fill=light-gray] (0, 0.8) -- (3, 2.3) -- (0,2.2) -- cycle;
		\draw[fill=light-gray] (0, 3.6) -- (3, 2.3) -- (1.5, 4.2) -- cycle;
		\draw[fill=light-gray] (3, 4.7) -- (0, 6.2) -- (1.5, 4.2) -- cycle;
		\draw[fill=light-gray] (3, 4.7) -- (3, 2.3) -- (5,3.6) -- cycle;
		\draw[-, line width=0.4mm, color=v] (1.5, 4.2) to (3, 2.3);
		\draw[-, line width=0.4mm, color=r] (1.5, 4.2) to (3, 4.7);
		\draw[-, line width=0.4mm, color=v] (1.5, 4.2) to (0, 6.2);
		\draw[-, line width=0.4mm, color=r] (1.5, 4.2) to (0, 3.6);
		\draw[-, line width=0.4mm, color=r] (-1.5, 4.2) to (0, 3.6);
		\draw[-, line width=0.4mm, color=v] (-1.5, 4.2) to (0, 6.2);
		\draw[-, line width=0.4mm, color=r] (-1.5, 4.2) to (-3, 4.7);
		\draw[-, line width=1mm, color=v] (-1.5, 4.2) to (-3, 2.3);
		\draw[-, line width=0.4mm, color=r] (0, 2.2) to (0, 3.6);
		\draw[-, line width=0.4mm, color=v] (0, 2.2) to (-3, 2.3);
		\draw[-, line width=0.4mm, color=r] (0, 2.2) to (0, 0.8);
		\draw[-, line width=1mm, color=v] (0, 2.2) to (3, 2.3);
		\draw[-, line width=0.4mm, color=v] (5, 3.6) to (3, 2.3);
		\draw[-, line width=0.4mm, color=r] (3, 4.7) to (5, 3.6);
		\draw[-, line width=1mm, color=v,rounded corners=10pt] (5, 3.6) -- (2.8, 6) -- (0, 6.2);
		\draw[-, line width=0.4mm, color=r, rounded corners=10pt] (-3, 4.7) -- (-1.5, 7.2) -- (5, 7.2) -- (5, 3.6);
		\draw[-, line width=0.4mm, color=r,rounded corners=10pt] (5, 3.6) -- (2.8, 0.5) -- (0, 0.8);
		\draw[-, line width=0.4mm, color=v, rounded corners=10pt] (-3, 2.3) -- (-1.5, -0.2) -- (5, -0.2) -- (5, 3.6);
		\path[every node/.style={font=\sffamily\small},color=e,line width=0.4mm]
		(-3, 4.7) edge node {} (0, 6.2)
		(3, 2.3) edge node {} (3, 4.7)
		(3, 2.3) edge node {} (0, 3.6)
		(0, 0.8) edge node {} (-3, 2.3)
		(3, 2.3) edge node {} (0, 0.8)
		(-3, 2.3) edge node {} (-3, 4.7)
		(0, 6.2) edge node {} (3, 4.7)
		(0, 3.6) edge node {} (-3, 2.3)
		(0, 6.2) edge node {} (0, 3.6);
		\draw [color=v,fill=v] (0, 0.8) circle [radius=0.3];
		\draw (0, 0.8) circle [radius=0.5];
		\draw [color=r,fill=r] (3, 2.3) circle [radius=0.3];
		\draw [color=r,fill=r] (-3, 2.3) circle [radius=0.3];
		\draw [color=v,fill=v] (3, 4.7) circle [radius=0.3];
		\draw [color=v,fill=v] (-3, 4.7) circle [radius=0.3];
		\draw (-3, 4.7) circle [radius=0.5];
		\draw [color=r,fill=r] (0, 6.2) circle [radius=0.3];
		\draw [color=v,fill=v] (0, 3.6) circle [radius=0.3];		
		\draw [color=e,fill=e] (5, 3.6) circle [radius=0.3];
		\draw (5, 3.6) circle [radius=0.5];
		\draw [color=e,fill=e] (1.5, 4.2) circle [radius=0.3];
		\draw [color=e,fill=e] (-1.5, 4.2) circle [radius=0.3];
		\draw (-1.5, 4.2) circle [radius=0.5];
		\draw [color=e,fill=e] (0, 2.2) circle [radius=0.3];
		\draw (0, 2.2) circle [radius=0.5];
		\end{tikzpicture}
	\end{center}
	\caption{An illustration for the proof of Theorem \ref{thm:planar_hypertrees_represent}. If $V-S$ is the set of circled violet nodes, then $\Gamma_{G_R}(V-S)$ is the set of circled emerald nodes, while the edges of $D_V$ leading from $V-S$ to $S$ corespond to the thick violet edges.}
	\label{f:hypertrees_repr_illustr}
\end{figure}

To finish the proof it is enough to show that the number of hypertrees of $G_R$ on $V$ is equal to the number of linear equivalence classes of $\Pic(D_V)$ of degree $|E|-1$.

As $\Pic(D_V)=\Pic^0(D_V)\times \mathbb{Z}$, the number of linear equivalence classes of $\Pic(D_V)$ of degree $|E|-1$ is equal to the degree of $\Pic^0(D_V)$. Because by Theorem \ref{thm:isomorphism_between_sandpile_groups} we have $\Pic^0(D_V)\cong \Pic^0(D_R)$, it is enough to show that the number of hypertrees of $G_R$ on $V$ is equal to the order of $\Pic^0(D_R)$. By Theorem \ref{thm:Jaeger_trees_represent_hypertrees}, the number of hypertrees of $G_R$ on $V$ is equal to the number of Jaeger trees of $G_R$ with base node $b_0$ and base edge $b_0b_1$.
By Proposition \ref{prop:Jaeger--arborescence}, the number of such Jaeger trees is equal to the number of arborescences of $D_R$ rooted at $r_0$ where $r_0$ is such that $r_0b_0b_1$ is a black triangle of the trinity.
Finally, by Fact \ref{fact:order_of_sandpile_group}, the order of $\Pic^0(D_R)$ is equal to the number of arborescences of $D_R$ rooted at $r_0$. This finishes the proof of the theorem.
\end{proof}

\begin{remark}
	For the case of (not necessarily planar) undirected graphs, there exists an effective procedure for finding a hypertree equivalent to a given chip configuration of degree $|E|-1$. See \cite{GSTsemibreak} (there the procedure is written for the more general case of metric graphs).
	However, the analogue of this procedure does not work for the case of balanced plane digraphs. It would be interesting to find an effective algorithm for this case.
\end{remark}

\section{Rotor-routing} \label{sec:rotor}

Baker and Wang proved that the rotor-routing action of a planar ribbon graph coincides with its Bernardi action \cite{Baker-Wang}. They also showed that for non-planar ribbon graphs, the two actions can be different, and conjectured \cite[Conjecture 7.2]{Baker-Wang} that for any non-planar ribbon graph, there exists a base point such that the rotor-routing and the Bernardi action with the given base point are different.

In this section we show that the identity of the rotor-routing and Bernardi actions carries over to the balanced plane digraph case. 
Let us repeat the definition of the rotor-routing action, and the rotor-routing game following \cite{Holroyd}.

Let $D$ be a ribbon digraph. For digraphs, a ribbon structure means a cyclic ordering of the in- and out-edges around each vertex. For an out-edge $e$ pointing away from a vertex $v$, by $e^+$ we mean the next out-edge of $v$ according to the ribbon structure. Let $v_0$ be a fixed vertex of $D$ (the \emph{root} or \emph{sink}).

A \emph{rotor configuration} is a function $\varrho$ that assigns to each
vertex $v\neq v_0$ an out-edge with tail $v$. We call $\varrho(v)$ the \emph{rotor} at $v$.
A configuration of the rotor-routing game is a pair $(x,\varrho)$, where $x\in\Div(D)$ and $\varrho$ is a rotor configuration.

Given a configuration $(x,\varrho)$, a \emph{routing} at vertex $v$ results in the  
configuration $(x', \varrho')$, where
$\varrho'$ is the rotor configuration with
$$
\varrho'(u) = \left\{\begin{array}{cl} \varrho(u) & \text{if $u\neq v$,}  \\
\varrho(u)^+ & \text{if $u=v$},
\end{array} \right.
$$
and $x'=x-\mathbf{1}_v+\mathbf{1}_{v'}$ where $v'$ is the head of $\varrho'(v)$.

A step of the rotor-routing game is to take a vertex with a positive number of chips, and perform a routing at that vertex.

The rotor-routing action (with root $v_0$) is the action of $\Pic^0(D)$ on the in-arborescences of $D$ rooted at $v_0$. We first define the action of chip configurations of the form $\mathbf{1}_v-\mathbf{1}_{v_0}$. 

For an in-arborescence $A$ rooted at $v_0$, the action of a chip configuration of the form $\mathbf{1}_v-\mathbf{1}_{v_0}$ is defined in the following way. We can think of $A$ as a rotor-configuration, since each vertex $v\neq v_0$ has exactly one out-edge in $A$. Play a rotor-routing game started from $(\mathbf{1}_v-\mathbf{1}_{v_0}, A)$ until the chip reaches $v_0$. In other words, in this moment, the configuration of the game will be $(\mathbf{0}, \varrho)$ for some $\varrho$. It is proved in \cite{Holroyd}, that we eventually reach such a configuration, moreover, the rotor configuration $\varrho$ will be another arborescence $A'$ at this moment. (Notice that in this situation, the game is deterministic. It can happen that during the game the rotor-configuration is not an arborescence in some steps, but when the chip eventually reaches $v_0$, it will be. See more in \cite{Holroyd}.) Then take $(\mathbf{1}_v-\mathbf{1}_{v_0})_{v_0} A = A'$.

Note that (equivalence classes of) chip configurations of the form $\mathbf{1}_v-\mathbf{1}_{v_0}$ generate $\Pic^0(D)$. The action of a general $x\in \Pic^0(D)$ is defined linearly. By \cite{Holroyd}, this is well-defined.

We will need the following technical result from \cite{alg_rotor}, that gives a more easily checkable condition for $x_{v_0}A=A'$.

\begin{prop}\cite[Proposition 3.16]{alg_rotor} \label{prop:rotor_action_char}
If $x$ is any chip configuration, $A$ and $A'$ are in-arborescences rooted at $v_0$, and we can get $(\mathbf{0}, A')$ from $(x, A)$ by playing a rotor-routing game, then $x_{v_0} A = A'$. \hfill\qedsymbol
\end{prop}

In the undirected (bidirected) case, an in-arborescence can be identified with an undirected spanning tree, hence the rotor-routing action with any root can act on the same set of objects. Hence it makes sense to ask if the rotor-routing action is independent of the root. For this problem, the answer is analogous to the case of the Bernardi action. By Chan et al. \cite{Chan15}, the rotor-routing action of an undirected ribbon graph is independent of the root if and only if the ribbon structure is planar. In \cite{rr_comp}, the compatibility with planar duality was also proved, moreover, Baker and Wang proved that for (undirected) planar ribbon graphs, the rotor-routing and the Bernardi actions coincide.
For general digraphs, it is not obvious whether one can pull the rotor-routing actions with different roots onto the same set of objects. Notice, however, that for balanced planar digraphs, we can associate a hypertree to any in-arborescence: For an in-arborescence $A$ of $D_V$, its dual tree in $G_V$ is an $R$-cut Jaeger tree $T$ with base point $r_0$ and base edge $r_0e_0$ where $r_0$ and $e_0$ are chosen such that $r_0e_0v_0$ is a white triangle. (This can be proved as Proposition \ref{prop:Jaeger--arborescence}.) $T$ realizes a hypertree $f^*\in B_E(G_V)$ on $E$ (in the undirected case, $f^*$ corresponds to a spanning tree of the dual graph). Now $f =d_{G_V}|_E - \mathbf{1} - f^*$ is a hypertree in $B_E(G_R)$.
Let us associate the hypertree $f$ to $A$.
This way we can once again pull the action with any root onto the same set of objects. It is easy to check that in the undirected planar ribbon graph case, the above construction associates to any arborescence $A$ the characteristic vector of the spanning tree we get by forgetting the orientations in $A$.

\begin{thm}\label{thm:rotor_Bernardi_the_same}
	Let $D$ be a balanced plane digraph. For any root $v_0$, chip-configu\-ra\-tion $x$ and hypertree $f\in B_E(G_R)$, $x_{v_0} f = \varphi_{V\to E}([-x]) \oplus f$, i.e. 
	the rotor-routing action on $D$ with root $v_0$ coincides with the Bernardi action of the inverse. Consequently, the rotor-routing action is independent of the root in the case of balanced plane digraphs.
\end{thm}

\begin{remark}
   Baker and Wang obtained that the Bernardi action agrees with the rotor-routing action in the case of plane graphs. The Bernardi action in our interpretation agrees with the Bernardi action of the inverse in the interpretation of Baker and Wang. That is the reason that in our interpretation the rotor-routing action agrees with the Bernardi action of the inverse.
\end{remark}

In the proof we will need the following technical claim.
\begin{claim}\label{cl:arborescences_one_by_one}
	For any two in-arborescences $A$ and $A'$, there exist a sequence of arborescences $A=A_0, A_1, \dots , A_k=A'$ such that $A_{i+1}$ is obtained from $A_i$ by adding and removing one edge.
\end{claim}
\begin{proof}
	In an in-arborescence rooted at $v_0$, each vertex different from $v_0$ has one out-edge. Let $W\subset V$ be the set of vertices that have different out-edge in $A$ and in $A'$. Let us introduce a partial order on $W$: Let $w_1\leq w_2$ if $w_1$ is reachable on a directed path from $w_2$ in $A$. Let $w$ be a maximal element with respect to this partial order. Suppose that the out-edge of $w$ in $A$ is $wv$ and the out-edge of $w$ in $A'$ is $wv'$. We claim that $A-wv +wv'$ is another in-arborescence, hence we can set $A_1=A-wv +wv'$ and continue similarly. 
	
	Since all out-degrees of $A-wv +wv'$ are correct, we only need to prove that the underlying undirected graph of $A-wv +wv'$ is a tree. Suppose for a contradiction that $A-wv +wv'$ has a cycle. Then, since $wv'\in A'$ and $A'$ is an arborescence, there must be an edge $w'u$ in this cycle that is not in $A'$. But as $w'u\in A$, we have $w'\in W$. But since $w$ is reachable from $w'$ in $A$, we have $w\leq w'$ which contradicts the fact that $w$ was a maximal element for $\leq$, hence indeed $A-wv +wv'$ is an in-arborescence and we have proved our claim.
\end{proof} 

\begin{proof}[Proof of Theorem \ref{thm:rotor_Bernardi_the_same}]
	Take the trinity obtained from $D$, and let us call $D=D_V$ in the followings. Fix our root $v_0$ (which is a violet node of the trinity). Let us fix a red node $b_0$ and an emerald node $b_1$ such that $v_0b_0b_1$ is a black triangle.
	
	
	To any $f, f'\in B_E(G_R)$, $d_{G_R}|_E-\mathbf{1} - f=d_{G_V}|_E-\mathbf{1} - f$ and $d_{G_R}|_E-\mathbf{1} - f'=d_{G_V}|_E-\mathbf{1} - f'$ are from $B_E(G_V)$. 
	Hence there exist unique R-cut Jaeger trees $T$ and $T'$ of $G_V$ (with starting edge $b_0b_1$) such that $d_{G_V}|_E-\mathbf{1} - f=f_E(T)$ and $d_{G_V}|_E-\mathbf{1} - f'=f_E(T')$ or in other words, $f=d_{G_V}|_E - \mathbf{1} - f_E(T)$ and $f'=d_{G_V}|_E - \mathbf{1} - f_E(T')$. By Proposition \ref{prop:Jaeger--arborescence} (with colors permuted), the dual spanning trees to $T$ and $T'$ in $D_V$ are two in-arborescences $A$, and $A'$ rooted at $v_0$.
	
	We need to show that if for some $x \in \Pic^0(D_V)$, $(x)_{v_0} A = A'$, that is, for hypertrees, $x_{v_0} (d_{G_V} - \mathbf{1} - f_E(T)) = (d_{G_V} - \mathbf{1} - f_E(T'))$, then $\varphi_{V \to E}([-x]) \oplus (d_{G_V} - \mathbf{1} - f_E(T)) = (d_{G_V} - \mathbf{1} - f_E(T'))$. By Theorem \ref{thm:compatibility_of_the_natural_action}, the latter is equivalent to $\varphi_{V \to E}([x]) \oplus f_E(T) = f_E(T')$.
	
	By Claim \ref{cl:arborescences_one_by_one} it is enough to consider the case when $A'$ can be obtained from $A$ by removing an arc and adding one. As $A$ and $A'$ are both in-arborescences, this means that an arc $vv'$ is removed, and an arc $vv''$ is added. Suppose that in the ribbon structure of $D_V$ the out-edges at $v$ follow each other in the order $vv'=vu_0, vu_1, \dots, vu_k=vv''$. Then by Proposition \ref{prop:rotor_action_char}, $[x]=[k\mathbf{1}_v - \mathbf{1}_{u_1} - \dots - \mathbf{1}_{u_k}]$, since by performing $k$ routings at $v$ from the configuration $(x,A)$, we arrive at $(\mathbf{0}, A')$, moreover, $A$ and $A'$ are both arborescences.
	
	Now let us find $\varphi_{V \to E}([x])$. This is the equivalence class of a $y$ such that $(x,y,0)\wapr (0,0,0)$. We can argue similarly as in the proof of Lemma \ref{lem:lin_ekv=>haromszog_ekv}. Let $r_ie_i$ be the edge of $G_V$ dual to $vu_i$ for $i=0,\dots k$. Then $r_0,e_0, r_1, e_1, \dots , r_k, e_k$ follow each other in this order on the boundary of the face of $G_V$ corresponding to $v$. (See Figure \ref{f:rotor_thm} for an example.) By taking the white triangle $vr_ie_{i-1}$ with coefficient one and the white triangles of the form $e_ir_iu_i$ with coefficient $-1$, for each $i=1, \dots, k$, we obtain $(x,\mathbf{1}_{e_0}-\mathbf{1}_{e_k}, 0)\wapr (0,0,0)$. Hence $\varphi_{V \to E}([x])=[\mathbf{1}_{e_0}-\mathbf{1}_{e_k}]$. 
	
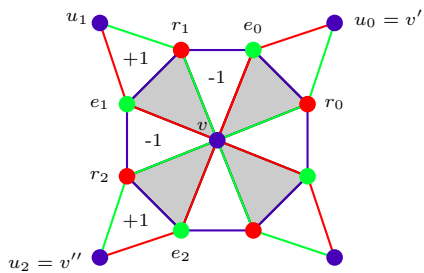
\begin{figure}
\begin{center}
	    \begin{tikzpicture}[-,>=stealth',auto,scale=1.2,
	    thick]
	    \draw[fill=light-gray] (0,0) -- (0.4,1) -- (1,0.4) -- cycle;
	    \draw[fill=light-gray] (0,0) -- (1,-0.4) -- (0.4, -1) -- cycle;
	    \draw[fill=light-gray] (0,0) -- (-0.4,-1) -- (-1,-0.4) -- cycle;
	    \draw[fill=light-gray] (0,0) -- (-1, 0.4) -- (-0.4, 1) -- cycle;    
	    \draw [thick,red] (0, 0) -- (0.4, 1);
	    \draw [thick,red] (0, 0) -- (1, -0.4);
	    \draw [thick,red] (0, 0) -- (-0.4, -1);
	    \draw [thick,red] (0, 0) -- (-1, 0.4);
	    \draw [thick,red] (0.4, 1) -- (1.3, 1.3);
	    \draw [thick,red] (1, -0.4) -- (1.3, -1.3);
	    \draw [thick,red] (-0.4, -1) -- (-1.3, -1.3);
	    \draw [thick,red] (-1, 0.4) -- (-1.3, 1.3);
	    \draw [thick,color=e] (0, 0) -- (-0.4,1);
	    \draw [thick,color=e] (0, 0) -- (1, 0.4);
	    \draw [thick,color=e] (0, 0) -- (0.4,-1);
	    \draw [thick,color=e] (0, 0) -- (-1, -0.4);
	    \draw [thick,color=e] (-0.4,1) -- (-1.3, 1.3);
	    \draw [thick,color=e] (1, 0.4) -- (1.3, 1.3);
	    \draw [thick,color=e] (0.4,-1) -- (1.3, -1.3);
	    \draw [thick,color=e] (-1, -0.4) -- (-1.3, -1.3);
	    \draw [thick,color=v] (-1, -0.4) -- (-1, 0.4);
	    \draw [thick,color=v] (-1, 0.4) -- (-0.4,1);
	    \draw [thick,color=v] (-0.4, 1) -- (0.4,1);
	    \draw [thick,color=v] (0.4, 1) -- (1, 0.4);
	    \draw [thick,color=v] (1, 0.4) -- (1, -0.4);
	    \draw [thick,color=v] (1, -0.4) -- (0.4,-1);
	    \draw [thick,color=v] (0.4, -1) -- (-0.4, -1);
	    \draw [thick,color=v] (-0.4, -1) -- (-1, -0.4);
	    \draw [thick,red] (0, 0) -- (0.4,1);
	    \tikzstyle{r}=[color=r,circle,draw,fill,font=\sffamily\small]
	    \tikzstyle{v}=[color=v,circle,draw,fill]
	    \tikzstyle{e}=[color=e,circle,draw,fill]
	    \tikzstyle{u}=[circle]
	    \draw [fill=v,v] (0, 0) circle [radius=0.08];
	    \draw [fill=v,v] (1.3, -1.3) circle [radius=0.08];
	    \draw [fill=v,v] (1.3, 1.3) circle [radius=0.08];
	    \draw [fill=v,v] (-1.3, 1.3) circle [radius=0.08];
	    \draw [fill=v,v] (-1.3, -1.3) circle [radius=0.08];
	    \draw [fill=e,e] (0.4, 1) circle [radius=0.08];
	    \draw [fill=e,e] (1, -0.4) circle [radius=0.08];
	    \draw [fill=e,e] (-0.4, -1) circle [radius=0.08];
	    \draw [fill=e,e] (-1, 0.4) circle [radius=0.08];
	    \draw [fill=r,r] (1, 0.4) circle [radius=0.08];
	    \draw [fill=r,r] (-1, -0.4) circle [radius=0.08];
	    \draw [fill=r,r] (-0.4, 1) circle [radius=0.08];
	    \draw [fill=r,r] (0.4, -1) circle [radius=0.08];
	    \node[u]  at (0, 0.7) {\scriptsize{-1}};
	    \node[u]  at (-0.7, 0) {\scriptsize{-1}};
	    \node[u]  at (-0.9, 0.9) {\scriptsize{+1}};
	    \node[u]  at (-0.9, -0.9) {\scriptsize{+1}};
	    \node[u]  at (-0.165, 0.16) {\scriptsize{$v$}};
	    \node[u]  at (1.9, 1.35) {\scriptsize{$u_0=v'$}};
	    \node[u]  at (-1.55, 1.35) {\scriptsize{$u_1$}};
	    \node[u]  at (-1.9, -1.35) {\scriptsize{$u_2=v''$}};
	    \node[u]  at (1.3, 0.4) {\scriptsize{$r_0$}};
	    \node[u]  at (0.4, 1.25) {\scriptsize{$e_0$}};
	    \node[u]  at (-0.4, 1.25) {\scriptsize{$r_1$}};
	    \node[u]  at (-1.3, 0.4) {\scriptsize{$e_1$}};
	    \node[u]  at (-1.3, -0.4) {\scriptsize{$r_2$}};
	    \node[u]  at (-0.4, -1.28) {\scriptsize{$e_2$}};
	    \end{tikzpicture}
\end{center}
\caption{An illustration for the proof of Theorem \ref{thm:rotor_Bernardi_the_same}.}
\label{f:rotor_thm}
\end{figure}

	Notice that $T'=T+e_0r_0 - e_kr_k$, hence $f_E(T')=f_E(T)+\mathbf{1}_{e_0}-\mathbf{1}_{e_k}$. Thus, indeed, $\varphi_{V \to E}([x]) \oplus f_E(T) = f_E(T')$, finishing the proof.
\end{proof}

\end{document}